\documentclass[sn-mathphys,smallextended]{sn-jnl}

\RequirePackage{fix-cm}

\usepackage{graphicx}
\usepackage{multirow}
\usepackage{subcaption}
\usepackage{hyperref}

\usepackage{anyfontsize}
\newcommand{\snorm}[1]{\|#1\|^2}
\newcommand{\bv}[1]{\mathbf{#1}}
\newcommand{\del}{\partial}

\jyear{2024}

\theoremstyle{thmstyleone}

\theoremstyle{thmstyletwo}

\newtheorem{remark}{Remark}

\theoremstyle{thmstylethree}
\newtheorem{definition}{Definition}

\newtheorem{lemma}{Lemma}[section]

\begin{document}

\title{Efficient discretization of the Laplacian on complex geometries}

\author{\fnm{Gustav} \sur{Eriksson}}\email{gustav.eriksson@it.uu.se}

\affil{\orgdiv{Department of Information Technology}, \orgname{Uppsala University}, \orgaddress{\street{PO Box 337}, \city{Uppsala}, \postcode{S-751 05}, \country{Sweden}}}

\abstract{Highly accurate simulations of problems including second derivatives on complex geometries are of primary interest in academia and industry. Consider for example the Navier-Stokes equations or wave propagation problems of acoustic or elastic waves. Current finite difference discretization methods are accurate and efficient on modern hardware, but they lack flexibility when it comes to complex geometries. In this work I extend the continuous summation-by-parts (SBP) framework to second derivatives and combine it with spectral-type SBP operators on Gauss-Lobatto quadrature points to obtain a highly efficient discretization (accurate with respect to runtime) of the Laplacian on complex domains. The resulting Laplace operator is defined on a grid without duplicated points on the interfaces, thus removing unnecessary degrees of freedom in the scheme, and is proven to satisfy a discrete equivalent to Green's first identity. Semi-discrete stability using the new Laplace operator is proven for the acoustic wave equation in 2D. Furthermore, the method can easily be coupled together with traditional finite difference operators using glue-grid interpolation operators, resulting in a method with great practical potential. Two numerical experiments are done on the acoustic wave equation in 2D. First on a problem with an analytical solution, demonstrating the accuracy and efficiency properties of the method. Finally, a more realistic problem is solved, where a complex region of the domain is discretized using the new method and coupled to the rest of the domain discretized using a traditional finite difference method.}


\maketitle

\section{Introduction}
It has long been known that high-order numerical methods are well suited for discretizing hyperbolic partial differential equations (PDE) \cite{Kreiss1972,doi:10.1137/0711076}. High-order methods not only provide accurate approximations for a relatively small number of degrees of freedom (DOFs), but they also lend themselves to efficient implementation on modern architecture due to improved cache utilization. A well-explored approach to constructing robust high-order methods is to use discretization schemes with summation-by-parts (SBP) properties. The concept originates from the finite difference (FD) community in the 70s, where finite difference operators were designed to mimic integration-by-parts in the discrete setting, which facilitates stability proofs for the discretized problem (e.g. the energy method). However, SBP operators are not limited to finite difference methods. The first SBP-FD methods were inspired by the continuous finite element method, and many commonly used numerical discretization methods can be cast in the form of SBP operators, for example finite volume methods \cite{NORDSTROM2003453,NORDSTROM2009875} and discontinuous Galerkin methods \cite{doi:10.1137/130932193,CHAN2018346,https://doi.org/10.1002/fld.3923}. Nevertheless, high-order finite difference methods in particular are very efficient for hyperbolic problems. Especially operators with diagonal norms (mass matrices) since these naturally lend themselves to explicit time-stepping schemes. How to construct high-order SBP finite difference schemes, and in particular how to impose boundary conditions, has garnered much attention in the past. Among the common methods are the simultaneous approximation term method (SBP-SAT) \cite{Carpenter1994,DelReyFernandez2014,Svard2014}, the projection method (SBP-P) \cite{Olsson1995a,Olsson1995,Mattsson2018}, and the ghost-point method (SBP-GP) \cite{CiCP-8-1074,doi:10.1137/20M1339702}.

Although diagonal-norm SBP finite difference operators suffer from a reduction in accuracy \cite{doi:10.1137/17M1139333} (for a $2p$th-order accurate interior stencil the boundary stencils can be at most of order $p$), it has been shown in a collection of recent papers that the error can be significantly reduced by moving a few grid points close to the boundaries \cite{MATTSSON201491}. These \emph{boundary-optimized} operators are very efficient (accurate with respect to computational costs) compared to traditional equidistant SBP operators for various PDEs, including the acoustic wave equation \cite{STIERNSTROM2023112376} and the compressible Euler equations \cite{Mattsson2018}. But, as for the traditional operators, they are only applicable to rectangular domains where the one-dimensional operators can be applied in all dimensions using tensor products. The usual tools to handle more complicated domains are: 1) mapping the PDE to a square reference domain and 2) decomposing the domain into quadrilateral blocks (possibly with curved boundaries) and coupling the blocks together using suitable interface conditions. Examples where these two approaches are used together to obtain high-order finite difference discretizations on complex domains can be found in \cite{Mattsson2007,TAZHIMBETOV2023112470,Almquist2020}. However, this approach can become problematic. Finding a block decomposition into quadrilaterals is non-trivial and, since both the traditional and boundary-optimized operators have a relatively high number of minimum grid points required to be defined, the blocks should be as large as possible. Otherwise, the total number of DOFs becomes unnecessarily large and the CFL condition demands very small time steps. Furthermore, the methods typically used to impose interface conditions, SBP-SAT, SBP-P, and SBP-GP, all require duplicated points on the interfaces. Therefore the number of duplicated points, which in a sense are superfluous DOFs, can become relatively large relative if the domain is decomposed into a large number of blocks (many interfaces).

An alternative to the traditional (equidistant) and boundary-optimized SBP finite difference operators for complex geometries is to use SBP operators defined on Gauss quadratures nodal points. These are a subclass of so-called generalized SBP (GSBP) operators, with relatively high accuracy on a few non-equidistant grid points. Extensive work has been done where such operators are derived and used to solve various problems, see for examples \cite{CARPENTER199674,doi:10.1137/120890144,doi:10.1137/140992205,DelReyFernandez2018}. In this work, I use GSBP operators defined on Gauss-Lobatto (GL) quadrature nodal points \cite{DELREYFERNANDEZ2014214}. For a given order, these are defined on a fixed and small number of non-equidistant grid points, including the endpoints, where the associated first derivative operator has non-repeating interior stencils. In my view, this makes them ideally suited for high-order discretizations on domains with many small blocks, for the following reasons: 
\begin{itemize}
	\item The total number of grid points per block for a given order is relatively small, which reduces the issues of unnecessarily refined regions in the domain.
	\item Since the interior stencils are non-repeating, higher-order derivative operators can be obtained by repeated use of the first derivative operators without the risk of spurious oscillations \cite{doi:10.1137/140992205}.
	\item Since the operators are defined on grids including the endpoints, the total number of DOFs in the scheme can be significantly reduced using the continuous SBP method \cite{doi:10.1137/15M1038360,doi:10.2514/6.2019-3206,Hicken2020} (or equivalently the embedding method \cite{olsson2024projections}) without any loss in accuracy, leading to a more efficient scheme.
\end{itemize}
Additionally, using SBP operators on the GL nodal points allows for straightforward computation of interpolation operators to other grid point distributions using the glue-grid framework \cite{Kozdon2016}. In \cite{doi:10.1137/22M1530690}, such interpolation operators are used to couple traditional SBP operators to a discontinuous Galerkin (DG) method for the acoustic wave equation. The idea in \cite{doi:10.1137/22M1530690} is to use the DG method in regions of the domain with complex geometries and the finite difference method elsewhere. In essence, the goal of this paper is to develop a scheme that solves the same problem, but with the added benefits of retaining diagonal-norms and not having to implement two different numerical methods. Furthermore, in \cite{doi:10.1137/22M1530690} the SBP-SAT method is used to impose both interface conditions, which has been shown to introduce additional stiffness to the scheme \cite{ERIKSSON2023111907,Eriksson2023}. Here I use a hybrid method where the continuity of the solutions is imposed strongly and the continuity of the normal derivatives weakly, thus avoiding this issue. The main contribution of this paper is a new SBP finite difference discretization of the Laplace operator that is efficient for many small quadrilateral blocks, is provably stable, and has a diagonal norm. Two numerical experiments are done on the acoustic wave equation. First an accuracy study on a problem with an analytical solution and then a realistic problem with real-world relevance to demonstrate the method's potential use, including a glue-grid interface coupling to traditional SBP operators.

The paper is structured as follows: in Section \ref{sec: prels} the notation and some definitions are presented, in Section \ref{sec: cont_SBP_lapl} the continuous SBP method for the Laplace operator is presented, in Section \ref{sec: aco_wave_eq} the discretization and stability proof of the acoustic wave equation is presented, in Section \ref{sec: num_res} the numerical experiments are done, and in Section \ref{sec: concl} conclusion are drawn.
\section{Preliminaries}
\label{sec: prels}
Consider a general domain $\Omega \in \mathbb{R}^d$, where $\partial \Omega$ denotes the boundary. Green's first identity states that, for two sufficiently smooth functions $u$ and $v$ defined on $\Omega$, the following holds:
\begin{equation}
	\label{eq: greens_first}
	(u,\Delta v)_\Omega = -(\nabla \cdot u, \nabla \cdot v)_\Omega + \langle u,\bv{n} \cdot \nabla v \rangle_{\partial \Omega},
\end{equation}
where $\bv{n}$ is the normal to the boundary and
\begin{equation}
	(u,v) = \int_\Omega u v \: d\bv{x} \quad \text{and} \quad \langle u,v \rangle_{\partial \Omega} = \int_{\partial \Omega} u v \: ds.
\end{equation}
The relationship \eqref{eq: greens_first} is useful when analyzing PDEs, it can for example be used to prove well-posedness of the heat equation and the acoustic wave equation. The numerical method for discretizing the Laplace operator is presented for $d = 2$ in this paper, but the method directly extends to $d = 3$. The approach taken is based on SBP operators, which leads to a discrete operator satisfying a discrete equivalent to \eqref{eq: greens_first}. 
\subsection{One-dimensional SBP operators}
The discretization of the two-dimensional Laplace operators will be constructed using one-dimensional SBP operators as building blocks, therefore I start with some necessary definitions in 1D. Let $I = [a,b]$ define an interval and let $\bv{x} = [x_1,x_2,...,x_m]$ be an ordered set of $m$ grid points (not necessarily equidistant) on $I$. In this work, I restrict myself to discretizations where $x_1 = a$ and $x_m = b$, i.e., the endpoints of the interval are represented by grid points. This includes the traditional and boundary-optimized finite difference SBP operators but also operators defined on the Gauss-Lobatto points that are used here. With this assumption, the boundary restriction operators, $e_l, e_r \in \mathbb{R}^m$ (operators restricting solution vectors to the boundaries), are particularly simple:
\begin{equation}
	e_l = [1,0,...,0]^\top \quad \text{and} \quad e_r = [0,...,0,1]^\top.
\end{equation}
SBP operators are associated with a norm matrix $H$, that defines the following inner product and norm:
\begin{equation}
	(\bv{u},\bv{v})_H = \bv{u}^\top H \bv{v} \quad \text{and} \quad \snorm{\bv{u}}_H = (\bv{u},\bv{u})_H, \quad \forall \bv{u},\bv{v} \in \mathbb{R}^m.
\end{equation}
In the case of Gauss-Lobatto SBP operators, $H$ is simply a diagonal matrix with the Gauss-Lobatto quadrature weights scaled to $I$ on the diagonal. The following definition specifies the SBP property of the first derivative operators, which will be useful later:
\begin{definition}
	\label{def: sbp_D1}
	A difference operator $D_1 \approx \frac{\del}{\del x}$ is said to be a first derivative SBP operator if, for the norm $H$, the relation
	\begin{equation}
		\label{eq: disc_sbp_prop_D1}
		H D_1 + D_1^\top H = -e_l e_l^\top + e_r e_r^\top,
	\end{equation}
	holds.
\end{definition}
The property \eqref{eq: disc_sbp_prop_D1} is the discrete equivalent to integration-by-parts. I will refer to the SBP operators as $p$th order accurate, meaning that they differentiate polynomials of order up to and including $p$ exactly. In the numerical experiments section I will present results for $p = 5$, $p = 7$, and $p = 9$, the explicit form of these operators can be found at \url{https://github.com/guer7/sbp_embed}.

Using the first derivative operators twice, a variable coefficient second derivative operator approximating $\frac{\del}{\del x}(b(x) \frac{\del}{\del x})$ is given by
\begin{equation}
	D_2^{(\bv{b})} = D_1 \bv{\bar{b}} D_1,
\end{equation}
where $\bv{b}$ is a vector of $b(x)$ evaluated in the grid points $\bv{x}$ and $\bv{\bar{b}}$ is a diagonal matrix with $\bv{b}$ on the diagonal. Constructing second derivative finite difference operators by applying the first derivative operator twice can have some drawbacks. For central finite difference operators with repeating interior stencils the resulting operator has a wide stencil and it fails to resolve the highest frequency mode on the grid, which can lead to spurious oscillations in the solution \cite{MATTSSON20082293}. The SBP GL operators used here are full matrices without repeating interior stencils and therefore no concept of wide stencils exists and spurious oscillations can not occur. Simultaneously, the SBP GL operators are very accurate for a relatively small number of grid points, which makes them ideally suited for discretizations of small blocks.
\subsection{Two-dimensional operators}
\label{subsec: two_dim_ops}
Consider a square reference domain $\tilde \Omega = [0,1] \times [0,1]$ with coordinates $\xi, \eta $. Using tensor products, the one-dimensional operators are extended to $\tilde \Omega$ as follows:
\begin{equation}
	\label{eq: disc_2d_tensor_prods}
	\begin{alignedat}{5}
		D_\xi &= (D_1 \otimes I_{m}), \quad &&D_\eta = (I_{m} \otimes D_1), \\
		H_\xi &= (H \otimes I_{m}), &&H_\eta = (I_{m} \otimes H), \\
		e_w &= (e_l^\top \otimes I_{m}), &&e_e = (e_r^\top \otimes I_{m}), \\
		e_s &= (I_{m} \otimes e_l^\top), &&e_n = (I_{m} \otimes e_r^\top), \\
	\end{alignedat}	
\end{equation}
where $I_m$ denotes the $m \times m$ identity matrix and $m$ is the number of grid points in each direction. Let also $N = m^2$ denote the total number of grid points in the block. The two-dimensional variable coefficient second derivative operators $D_{\xi \xi}^{(\bv{b})}$ and $D_{\eta \eta}^{(\bv{b})}$, approximating $\frac{\del}{\del \xi}(b(\xi,\eta) \frac{\del}{\del \xi})$ and $\frac{\del}{\del \eta}(b(\xi,\eta) \frac{\del}{\del \eta})$ respectively, can not be directly constructed using tensor products as in \eqref{eq: disc_2d_tensor_prods}. Instead, the one-dimensional operators are built line-by-line with the corresponding values of $\bv{b}$ and stitched together to form the two-dimensional operators. See \cite{Almquist2014} for more details on this. 

As previously mentioned, the two-dimensional operators derived above are defined on the rectangular reference domain $\tilde \Omega$. More complicated geometries are treated using curvilinear mappings. With this approach, the spatial derivatives in the PDE (expressed in terms of the coordinates in $\Omega$) are rewritten in terms of the coordinates in $\tilde \Omega$, and the mapped PDE is discretized using the SBP operators defined on the reference domain. Discretizations of the Laplace operators on curvilinear grids based on SBP finite difference operators have been presented in the literature many times. Here I simply restate some of the results needed for the upcoming analysis. For details on how to derive the transformed operators, see for example \cite{Almquist2014,Almquist2020,STIERNSTROM2023112376}.

Let $x = x(\xi,\eta)$ and $y = y(\xi,\eta)$ define a mapping from the physical domain to the reference domain $\tilde \Omega$ and let $\bv{x}, \bv{y} \in \mathbb{R}^{N}$ be vectors containing the grid points on the physical grid. A Laplace operator defined on the physical domain $\Omega$ is then given by
\begin{equation}
	\label{eq: disc_lapl_def}
	D_L = J^{-1} (D_{\xi \xi}^{(\alpha_1)} + D_\eta \beta D_\xi + D_\xi \beta D_\eta + D_{\eta \eta} ^{(\alpha_2)}),
\end{equation}
where 
\begin{equation}
	\label{eq: disc_metric_coeffsa}
	\begin{alignedat}{2}
		J &= X_\xi Y_\eta - {X}_\eta {Y}_\xi, \\
		\alpha_1 &= {J}^{-1} ({X}_\eta^2 + {Y}_\eta^2), \\
		\beta &= -{J}^{-1} ({X}_\xi {X}_\eta + {Y}_\xi {Y}_\eta), \\
		\alpha_2 &= {J}^{-1} ({X}_\xi^2 + {Y}_\xi^2),
	\end{alignedat}
\end{equation}
and
\begin{equation}
	\label{eq: disc_metric_dervs}
	\begin{alignedat}{2}
		{X}_\xi = \text{diag}(D_\xi \bv{x}), \\
		{X}_\eta = \text{diag}(D_\eta \bv{x}), \\
		{Y}_\xi = \text{diag}(D_\xi \bv{y}), \\
		{Y}_\eta = \text{diag}(D_\eta \bv{y}).
	\end{alignedat}
\end{equation}
The first derivative operators in the $x$- and $y$-directions are given by
\begin{equation}
	\begin{alignedat}{2}
		D_x &= J^{-1} ({Y}_\eta D_\xi - {Y}_\xi D_\eta), \\
		D_y &= J^{-1} (-{X}_\eta D_\xi + {X}_\xi D_\eta).
	\end{alignedat}
\end{equation}
The discrete inner product and boundary quadratures are given by
\begin{equation}
	\begin{alignedat}{2}
		{H}_\Omega &= H_\xi H_\eta J , \\
		{H_w} &=  H e_w {W}_2 e_w^\top,  \\
		{H_e} &= H e_e {W}_2 e_e^\top , \\
		{H_s} &= H e_s {W}_1 e_s^\top, \\
		{H_n} &= H e_n {W}_1 e_n^\top.
	\end{alignedat}
\end{equation}
where
\begin{equation}
	{W}_1 = \sqrt{{X}_\xi^2 + {Y}_\xi^2} \quad \text{and} \quad {W}_2 = \sqrt{{X}_\eta^2 + {Y}_\eta^2},
\end{equation}
are boundary scaling factors. The Laplace operator \eqref{eq: disc_lapl_def} can be shown to satisfy 
\begin{equation}
	\label{eq: disc_laplace_sbp}
	\begin{alignedat}{2}
		(\bv{u},D_L \bv{v})_{H_\Omega} &= - (D_x \bv{u}, D_x \bv{v})_{H_\Omega} - (D_y \bv{u}, D_y \bv{v})_{H_\Omega} \\
		&+ \sum_{k = w,e,s,n} \langle e_k \bv{u}, d_k \bv{v} \rangle_{H_k} , \quad \forall \bv{u},\bv{v} \in \mathbb{R}^{N},
	\end{alignedat}
\end{equation}
where 
\begin{equation}
	d_k = n_k^{(1)} e_k D_x + n_k^{(2)} e_k D_y, \quad k = w,e,s,n,
\end{equation}
approximates the normal derivatives on the boundaries. Here $n_k^{(1)}$ and $n_k^{(2)}$ are diagonal matrices with the $x$- and $y$-components of the normal vectors in the grid points at the boundaries. The discrete property \eqref{eq: disc_laplace_sbp} is the discrete equivalent to Green's first identity on a single curvilinear block.

For many geometries, it is not possible to construct a well-defined mapping from a square domain (for example the interior of a circle). In such cases, to discretize a PDE using these types of finite difference operators, one has to decompose the domain into smaller blocks, map each block to a square reference domain, and couple the blocks together. In the continuous setting, it is well known that the function and its normal derivative must be continuous at the interface between two blocks to retain Green's first identity \eqref{eq: greens_first}. But how to impose these \emph{interface conditions} in the discrete setting in a way such that a similar property to \eqref{eq: disc_laplace_sbp} is obtained in the whole domain is non-trivial, and has received a lot of attention in the past. The SBP-SAT method is the most studied \cite{Almquist2020,Wang2019,Mattsson2008}, but works using the SBP-P can also be found in the literature \cite{Mattsson2006}. Recently, in \cite{Eriksson2023,eriksson2023acoustic}, a hybrid SBP-P-SAT method was used, where continuity of the normal derivatives is imposed using SAT and continuity of the solution using the projection method. However, in all of the works cited above the resulting discretizations are defined on grids where the points on the interfaces are stored twice (once per block). This is no issue when relatively few blocks are needed to discretize the domain, but for domains where many small blocks are needed, it leads to a large number of unnecessary degrees of freedom. The goal of this paper is to derive a scheme that avoids this issue using the continuous SBP method.
\section{The continuous SBP Laplace operator}
\label{sec: cont_SBP_lapl}
The discretization method presented in this paper is numerically similar to the SBP-P-SAT discretizations presented in \cite{Eriksson2023,eriksson2023acoustic} for the acoustic wave equation and in \cite{ERIKSSON2023111907} for the dynamic beam equation, but instead of using the projection method to impose continuity across the interfaces, the continuous SBP method is used. The continuous SBP method was first presented in \cite{doi:10.1137/15M1038360} for first-order problems on simplex elements. In \cite{olsson2024projections} the same method is developed but presented slightly differently using an embedding operator (referred to as the embedding method). For completeness and notational clarity, I will present the method below following the notation in \cite{olsson2024projections}.

Consider a decomposition of a general two-dimensional domain into quadrilateral blocks (possibly with curved sides to handle physical boundaries). Discretize each block as described in Section \ref{sec: prels} and stack the grid points of all blocks in a matrix $\bv{S} \in \mathbb{R}^{N\times2}$ as follows:
\begin{equation}
	\label{eq: disc_glob_gridpoints}
	\bv{S} = 
	\begin{bmatrix}
		\bv{S}_1 \\ \bv{S}_2 \\ \vdots \\ \bv{S}_{N_b}
	\end{bmatrix},
\end{equation}
where $N_b$ denotes the number of blocks and 
\begin{equation}
	\bv{S}_i = 
	\begin{bmatrix}
		x_1^{(i)} & y_1^{(i)} \\
		x_2^{(i)} & y_2^{(i)} \\
		\vdots & \vdots \\
		x_{N_i}^{(i)} & y_{N_i}^{(i)}
	\end{bmatrix},
\end{equation}
is the coordinates of block $i$, with $N_i$ grid points. The total number of grid points is given by $N = \sum_{i = 1}^{N_b} N_i$. A discrete global solution vector $\bv{v} \in \mathbb{R}^{N}$ is defined in a similar fashion:
\begin{equation}
	\bv{v} = 
	\begin{bmatrix}
		\bv{v}_1 \\ \bv{v}_2 \\ \vdots \\ \bv{v}_{N_b}
	\end{bmatrix},
	\quad \text{where} \quad 
	\bv{v}_i = 
	\begin{bmatrix}
		v_1^{(i)} \\
		v_2^{(i)} \\
		\vdots \\
		v_{m_i}^{(i)} &
	\end{bmatrix}.
\end{equation}
Assuming that the discretization points in each block conform to the boundary, some of the rows in $\bv{S}$ (those corresponding to grid points at the interfaces) will be identical. Therefore, since the physics requires that the solution is continuous at these points, the global solution vector $\bv{v}$ is unnecessarily large. It suffices to only store the values in grid points at the interfaces once without losing any information. Removing these superfluous DOFs is one of the main benefits of the continuous SBP method.

Let $\hat N$ denote the number of unique grid points (unique rows in $S$) and $\hat{\bv{S}} \in \mathbb{R}^{\hat N \times 2}$ the coordinate matrix without duplicate grid points. Define a matrix $E \in \mathbb{R}^{N \times \hat N}$ such that
\begin{equation}
	\label{eq: disc_E_prop}
	\bv{S} = E \hat{\bv{S}},
\end{equation}
where each row $i$ of $E$ is only non-zero at the index $j$ corresponding to the row in $\hat{\bv{S}}$ that equals the $i$:th row in $\bv{S}$, where it is one. The operation $E \hat{\bv{S}}$ essentially duplicates the rows in $\hat{\bv{S}}$ corresponding to grid points at the interfaces and places them in the order defined to $\bv{S}$. The matrix $E$ is referred to as the \emph{embedding operator} and shall be used to construct a reduced size multiblock Laplace operator.
\begin{remark}
	The continuous SBP method is not limited to the SBP GL operators. The ideas presented here are equally applicable to other SBP operators on discretizations with grid points on the boundaries, including traditional and boundary-optimized SBP operators.
\end{remark}

To make the presentation of the method easier to follow, I will consider the two-block domain $\Omega = \Omega^{(l)} \cup \Omega^{(r)}$ in Figure \ref{fig: mb_example} as an illustrating example. It should be clear from this example how to use the method on more complicated domain decompositions. I will use superscripts $(l)$ and $(r)$ to indicate operators and vectors in the left and right blocks, respectively. Discretize $\Omega^{(l)}$ and $\Omega^{(r)}$ according to Section \ref{subsec: two_dim_ops} and let
\begin{equation}
	H^+_\Omega = \begin{bmatrix}
		H_{\Omega^{(l)}} & 0 \\ 0 &H_{\Omega^{(r)}}
	\end{bmatrix},
\end{equation}
define a global inner product matrix, where the ''$+$'' sign is used to indicate that $H^+_\Omega$ is a non-reduced operator, i.e., it is defined for solution vectors with duplicated points. Similarly, let
\begin{equation}
	D_L^+ = \begin{bmatrix}
			D_L^{(l)} & 0 \\ 0 & D_L^{(r)}
	\end{bmatrix},
\end{equation}
define the non-reduced Laplace operator and
\begin{equation}
	D_x^+ = \begin{bmatrix}
			D_x^{(l)} & 0 \\ 0 & D_x^{(r)}
	\end{bmatrix} \quad \text{and} \quad 
	D_y^+ = \begin{bmatrix}
			D_y^{(l)} & 0 \\ 0 & D_y^{(r)}
	\end{bmatrix},
\end{equation}
$x$- and $y$-derivative operators. For later use, the physical boundary is split into three parts $\partial \Omega_1$, $\partial \Omega_2$, and $\partial \Omega_3$, see Figure \ref{fig: mb_example}, with the following boundary operators:
\begin{equation}
	\label{eq: disc_example_boundary_ops}
	\begin{alignedat}{5}
		&e_{\partial \Omega_1} = 
		\begin{bmatrix}
			e_n^{(l)} \\
			e_n^{(r)}
		\end{bmatrix} E, \quad 
		&&d_{\partial \Omega_1} = 
		\begin{bmatrix}
			d_n^{(l)} \\
			d_n^{(r)}
		\end{bmatrix} E, \quad 
		&&H_{\partial \Omega_1}= 
		\begin{bmatrix}
			H_n^{(l)} & 0 \\
			0 & H_n^{(r)}
		\end{bmatrix}, \\
		&e_{\partial \Omega_2} = 
		\begin{bmatrix}
			e_s^{(l)} \\
			e_s^{(r)}
		\end{bmatrix} E, 
		&&d_{\partial \Omega_2} = 
		\begin{bmatrix}
			d_s^{(l)} \\
			d_s^{(r)}
		\end{bmatrix} E, 
		&&H_{\partial \Omega_2} = 
		\begin{bmatrix}
			H_s^{(l)} & 0 \\
			0 & H_s^{(r)}
		\end{bmatrix}, \\
		&e_{\partial \Omega_3} = 
		\begin{bmatrix}
			e_w^{(l)} \\
			e_e^{(r)}
		\end{bmatrix} E, 
		&&d_{\partial \Omega_3} = 
		\begin{bmatrix}
			d_w^{(l)} \\
			d_e^{(r)}
		\end{bmatrix} E, 
		&&H_{\partial \Omega_3} = 
		\begin{bmatrix}
			H_w^{(l)} & 0 \\
			0 & H_e^{(r)}
		\end{bmatrix}
	\end{alignedat}
\end{equation}
Finally, the following non-reduced modified Laplace operator is defined:
\begin{equation}
	\tilde{D}_L^+ = D_L^+ - (H^+_\Omega)^{-1}
		\begin{bmatrix}
			e^{(l)\top}_e H^{(l)}_e d^{(l)}_e & e^{(l)\top}_e H^{(l)}_e d^{(r)}_w \\ 0 & 0
		\end{bmatrix},
\end{equation}
where continuity of the normal derivative across the interface is imposed using the SAT method. 
\begin{remark}
	\label{rmrk: bound_integral_eq}
	Note that if the metric coefficients \eqref{eq: disc_metric_dervs} in the left and right blocks are computed using the same SBP operators on the reference domain, the boundary integrals $H^{(l)}_e$ and $H^{(r)}_w$ are equal. This holds because the coordinates of the grid points on the interface are the same and only tangential derivatives on the reference domain are used when computing the scaling factors $W_2^{(l)}$ and $W_2^{(r)}$.
\end{remark}
\begin{figure}
	\center
	\includegraphics[width=0.9\textwidth]{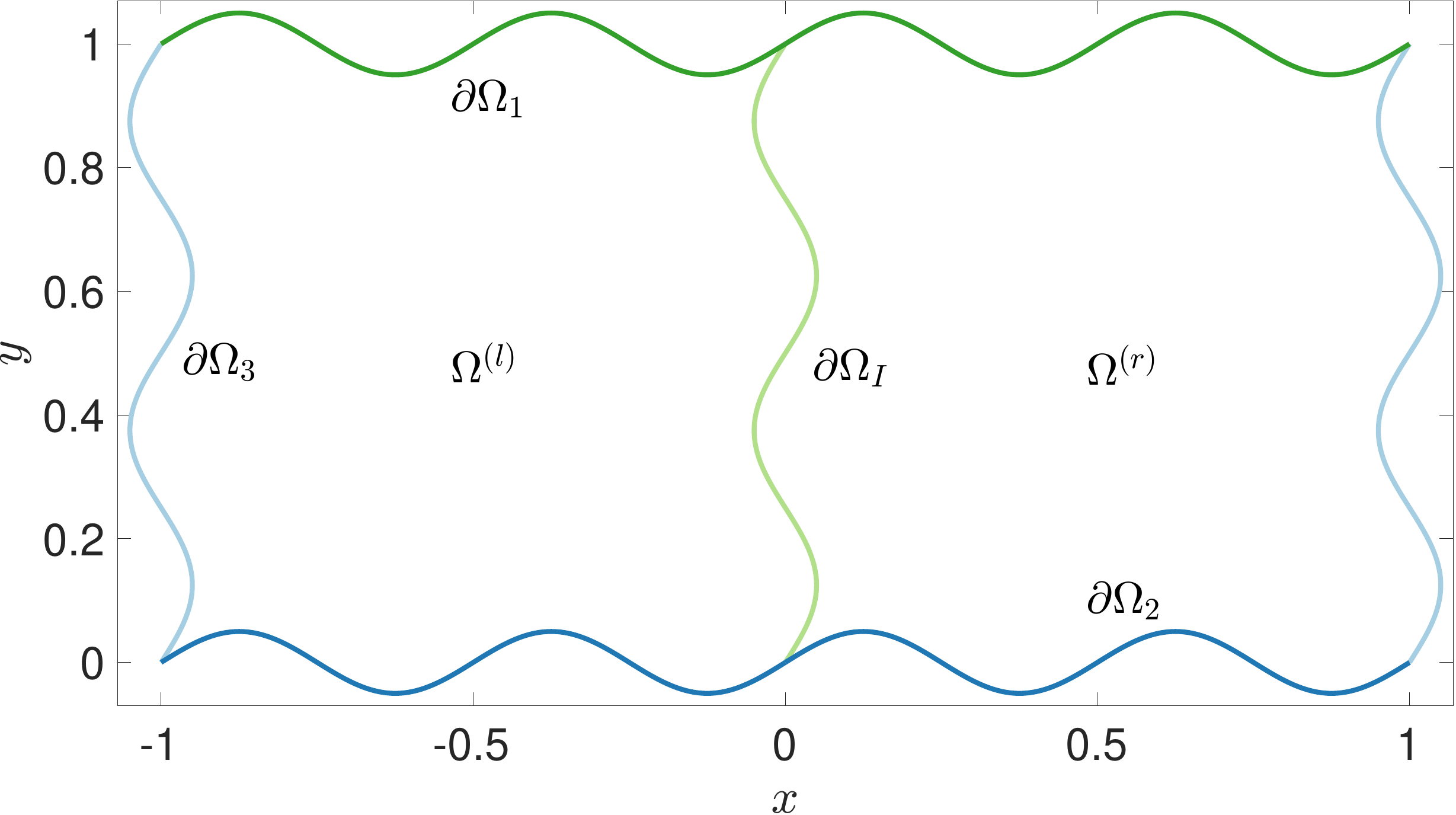}
	\caption{Two-block curvilinear example domain. The boundaries are color-coded as follows: the interface $\partial \Omega_I$ is light green, the north boundaries $\partial \Omega_1$ are dark green, the south boundaries $\partial \Omega_2$ are dark blue, and the side boundaries $\partial \Omega_3$ are light blue.}
	\label{fig: mb_example}
\end{figure}
The following lemma is one of the main results of the present study:
\begin{lemma}
	\label{lemma: lapl_sbp_prop}
	The reduced Laplace operator
	\begin{equation}
		\label{eq: red_Laplace_op}
		D_L = H_\Omega^{-1} E^T H_\Omega^+ \tilde{D}_L^+ E,
	\end{equation}
	where
	\begin{equation}
		\label{eq: red_inner_prod}
		H_\Omega = E^\top H^+_\Omega E,
	\end{equation}
	satisfies the SBP property
	\begin{equation}
		\label{eq: disc_laplace_sbp_embed}
		\begin{alignedat}{2}
			(\bv{u},D_L \bv{v})_{H_\Omega} &= - (D^+_x E \bv{u}, D^+_x E \bv{v})_{H^+_\Omega} - (D^+_y E \bv{u}, D^+_y E \bv{v})_{H^+_\Omega} \\
			&+ \langle e_{\partial \Omega} \bv{u}, d_{\partial \Omega} \bv{v} \rangle_{H_{\partial \Omega}}, \quad \forall \bv{u},\bv{v} \in \mathbb{R}^{\hat N},
		\end{alignedat}
	\end{equation}
	where the boundary terms are given by
	\begin{equation}
		\begin{alignedat}{2}
			 \langle e_{\partial \Omega} \bv{u}, d_{\partial \Omega} \bv{v} \rangle_{H_{\partial \Omega}} &= \sum_{i = 1,2,3} \langle e_{\partial \Omega_i} \bv{u}, d_{\partial \Omega_i} \bv{v} \rangle_{H_{\partial \Omega_i}}.
		\end{alignedat}
	\end{equation}
\end{lemma}
\begin{proof}
	We have
	\begin{equation}
		\label{eq: disc_lapl_sbpprop_1}
		(\bv{u},D_L \bv{v})_{H_\Omega} = (\bv{u}^+, \tilde{D}_L^+ \bv{v}^+)_{H^+_\Omega},
	\end{equation}
	where $\bv{u}^+ = E \bv{u} = \begin{bmatrix} \bv{u}^{(l)} \\ \bv{u}^{(r)} \end{bmatrix}$ and $\bv{v}^+ = E \bv{v} = \begin{bmatrix}
	\bv{v}^{(l)} \\ \bv{v}^{(r)} \end{bmatrix}$. Note that 
	\begin{equation}
		\label{eq: disc_lapl_sbpprop_2}
		e_e^{(l)} \bv{u}^{(l)} = e_w^{(r)} \bv{u}^{(r)} \quad \text{and} \quad e_e^{(l)} \bv{v}^{(l)} = e_w^{(r)} \bv{v}^{(r)},
	\end{equation}
	hold exactly due to the specific construction of $E$ (since it duplicates the values from one side of the interface to the other). Using \eqref{eq: disc_laplace_sbp} in \eqref{eq: disc_lapl_sbpprop_1}, we get
	\begin{equation}
		\begin{alignedat}{2}
		(\bv{u},D_L \bv{v})_{H_\Omega} &= - (D^{(l)}_x \bv{u}^{(l)}, D^{(l)}_x \bv{v}^{(l)})_{H_{\Omega^{(l)}}} - (D^{(l)}_y \bv{u}^{(l)}, D^{(l)}_y \bv{v}^{(l)})_{H_{\Omega^{(l)}}} \\
		&- (D^{(r)}_x \bv{u}^{(r)}, D^{(r)}_x \bv{v}^{(r)})_{H_{\Omega^{(r)}}} - (D^{(r)}_y \bv{u}^{(r)}, D^{(r)}_y \bv{v}^{(r)})_{H_{\Omega^{(r)}}} \\
		&+ \sum_{k = w,e,s,n} \langle e_k^{(l)} \bv{u}^{(l)}, d_k^{(l)} \bv{v}^{(l)} \rangle_{H^{(l)}_k} + \langle e_k^{(r)} \bv{u}^{(r)}, d_k^{(r)} \bv{v}^{(r)} \rangle_{H^{(r)}_k} \\
		&- \langle e_e^{(l)} \bv{u}^{(l)}, d_e^{(l)} \bv{v}^{(l)} + d_w^{(r)} \bv{v}^{(r)} \rangle_{H^{(l)}_e},
		\end{alignedat}
	\end{equation}
	where the first three rows come from $D_L^{(l)}$ and $D_L^{(r)}$ and the final row from the SAT. Expanding the sums over the boundaries, using \eqref{eq: disc_lapl_sbpprop_2}, and using that $H^{(l)}_e = H^{(r)}_w$ (see Remark \ref{rmrk: bound_integral_eq}) cancels the interface terms, resulting in
	\begin{equation}
		\begin{alignedat}{2}
		(\bv{u},D_L \bv{v})_{H_\Omega} &= - (D^{(l)}_x \bv{u}^{(l)}, D^{(l)}_x \bv{v}^{(l)})_{H_{\Omega^{(l)}}} - (D^{(l)}_y \bv{u}^{(l)}, D^{(l)}_y \bv{v}^{(l)})_{H_{\Omega^{(l)}}} \\
		&- (D^{(r)}_x \bv{u}^{(r)}, D^{(r)}_x \bv{v}^{(r)})_{H_{\Omega^{(r)}}} - (D^{(r)}_y \bv{u}^{(r)}, D^{(r)}_y \bv{v}^{(r)})_{H_{\Omega^{(r)}}} \\
		&+ \sum_{k = w,s,n} \langle e_k^{(l)} \bv{u}^{(l)}, d_k^{(l)} \bv{v}^{(l)} \rangle_{H^{(l)}_k} + \sum_{k = e,s,n} \langle e_k^{(r)} \bv{u}^{(r)}, d_k^{(l)} \bv{v}^{(r)} \rangle_{H^{(r)}_k}.
		\end{alignedat}
	\end{equation}
	Combining the terms leads to
	\begin{equation}
		\label{eq: disc_laplace_sbp_embed_proof}
		\begin{alignedat}{2}
			(\bv{u},D_L \bv{v})_{H_\Omega} &= - (D^+_x E \bv{u}, D^+_x E \bv{v})_{H_\Omega^+} - (D^+_y E \bv{u}, D^+_y E \bv{v})_{H_\Omega^+} \\
			&+ \langle e_{\partial \Omega} \bv{u}, d_{\partial \Omega} \bv{v} \rangle_{H_{\partial \Omega}}, \quad \forall \bv{u},\bv{v} \in \mathbb{R}^{\hat N},
		\end{alignedat}
	\end{equation}
	which concludes the proof.
\end{proof}
Although Lemma \ref{lemma: lapl_sbp_prop} proves the SBP property of the reduced Laplace operator for the simple two-block domain in Figure \ref{fig: mb_example}, the proof extends directly to more complicated domains. The steps to create the reduced Laplace operator can be summarized as follows:
\begin{enumerate}
	\item Decide on an ordering of the blocks, stack the coordinates (create $S$), and find a matrix $E$ satisfying \eqref{eq: disc_E_prop}.
	\item Construct the non-reduced inner product matrix $H_\Omega^+$, Laplace operator $D_L^+$, and corresponding boundary operators $e_{\partial \Omega}$, $d_{\partial \Omega}$, and $H_{\partial \Omega}$.
	\item Construct the modified Laplace operator by adding SATs imposing continuity of the fluxes across the interfaces.
	\item Construct the reduced inner product operator $H_\Omega$ according to \eqref{eq: red_inner_prod}.
	\item Construct the reduced Laplace operator $D_L$ according to \eqref{eq: red_Laplace_op}.
\end{enumerate}
The resulting Laplace operator is defined on a grid without duplicated points, with continuity of the solution imposed using the continuous SBP method and continuity of the normal derivatives imposed using the SAT method, that satisfies the following discrete equivalent to Green's first identity \eqref{eq: greens_first}:
\begin{equation}
	\label{eq: disc_laplace_sbp_embed_act}
	\begin{alignedat}{2}
		(\bv{u},D_L \bv{v})_{H_\Omega} &= - (D^+_x E \bv{u}, D^+_x E \bv{v})_{H_\Omega^+} - (D^+_y E \bv{u}, D^+_y E \bv{v})_{H_\Omega^+} \\
		&+ \langle e_{\partial \Omega} \bv{u}, d_{\partial \Omega} \bv{v} \rangle_{H_{\partial \Omega}}, \quad \forall \bv{u},\bv{v} \in \mathbb{R}^{\hat N},
	\end{alignedat}
\end{equation}
where $e_{\partial \Omega}$ extracts the solution, $d_{\partial \Omega}$ approximates the normal derivative, and $H_{\partial \Omega}$ is the boundary quadrature, all along the physical boundary $\partial \Omega$.
\section{The acoustic wave equation}
\label{sec: aco_wave_eq}
As a model problem to evaluate the new method I consider the acoustic wave equation with wave speed $c$ on $\Omega \subset \mathbb{R}^2$ given by,
\begin{equation}
	\label{eq: cont_wave_eq}
	\begin{alignedat}{5}
		&u_{tt} = c^2 \Delta u + \delta(x-x_s,y-y_s) f(t), \quad &&x,y \in \Omega, &&t > 0, \\
		&u = 0, &&x,y \in \partial \Omega_1, \quad &&t > 0, \\
		&n \cdot \nabla u = 0, &&x,y \in \partial \Omega_2, &&t > 0, \\
		&u_t + c n \cdot \nabla u = 0, &&x,y \in \partial \Omega_3, &&t > 0, \\
		&u = u_t = 0, &&x,y \in \Omega, &&t = 0, \\
	\end{alignedat}
\end{equation}
with Dirichlet, Neumann, and first-order outflow boundary conditions \cite{engquist} on $\partial \Omega_1$, $\partial \Omega_2$, and $\partial \Omega_3$, respectively. Here $\partial \Omega = \partial \Omega_1 \cup \partial \Omega_2 \cup \partial \Omega_3$ is the boundary of $\Omega$. The term $\delta(x-x_s,y-y_s) f(t)$ corresponds to a point source at $(x_s,y_s)$, where the function $f(t)$ is the time-dependent component. Multiplying \eqref{eq: cont_wave_eq} by $u_t$ and integrating over the domain results in
\begin{equation}
	(u_t, u_{tt})_\Omega = (u_t,c^2\Delta u + \delta(x-x_s,y-y_s) f(t))_\Omega.
\end{equation}
Using Green's first identity \eqref{eq: greens_first} and only considering the energy evolution of the homogeneous problem, i.e. $f(t) = 0$, gives
\begin{equation}
	\label{eq: cont_wave_eneq}
	\frac{d}{dt} \left ( \snorm{u_t}_\Omega + c^2 \snorm{\nabla \cdot u}_\Omega \right ) = 2 c^2 \langle u_t, \bv{n} \cdot \nabla u \rangle_{\partial \Omega} = -2c \snorm{u_t}_{\partial \Omega_3} \leq 0, 
\end{equation}
where 
\begin{equation}
	\snorm{u}_\Omega = (u,u)_\Omega \quad \text{and} \quad \snorm{u}_{\partial \Omega} = \langle u, u \rangle_{\partial \Omega}.
\end{equation}
The energy balance equation \eqref{eq: cont_wave_eneq} shows that the energy in the system given by $\snorm{u_t}_\Omega + c^2 \snorm{\nabla \cdot u}_\Omega$ can not grow (except for the contribution of the source term), and that it will dissipate through the outflow boundary $\partial \Omega_3$.

In this work, I will use the projection method to impose the Dirichlet boundary condition and the SAT method to impose the Neumann and outflow boundary conditions. These methods have been used previously to impose boundary conditions for the acoustic wave equation, in for example \cite{Mattsson2009,Mattsson2018}, but are here presented for completeness. 

Once again, for clarity, I present the discretization of these boundary conditions for the example two-block domain in Figure \ref{fig: mb_example}, where $\partial \Omega_1 = \partial \Omega^{(l)}_{n} \cup \partial \Omega^{(r)}_{n}$ (northern boundary), $\partial \Omega_2 = \partial \Omega^{(l)}_{s} \cup \partial \Omega^{(r)}_{s}$ (southern boundary), and $\partial \Omega_3 = \partial \Omega^{(l)}_{w} \cup \partial \Omega^{(r)}_{e}$ (side boundaries). Discretizing in space and leaving time continuous, the following semi-discrete equations are obtained:
\begin{equation}
	\label{eq: disc_wave_eq}
	\begin{alignedat}{2}
		&\bv{v}_{tt} = A v + B v_t + \bv{F}(t), \quad &&t > 0, \\
		&\bv{v} = \bv{v}_t = 0, &&t = 0, \\
	\end{alignedat}
\end{equation}
where
\begin{equation}
	\begin{alignedat}{2}
		A &= c^2 P (D_L + SAT_{\partial \Omega_2} + SAT^{(A)}_{\partial \Omega_3}) P, \\
		B &= c P SAT^{(B)}_{\partial \Omega_3} P, \\
		\bv{F}(t) &= P \bv{d}_s f(t).
	\end{alignedat}
\end{equation}
The SATs are given by
\begin{equation}
	\label{eq: disc_wave_bc_sats}
	\begin{alignedat}{2}
		SAT_{\partial \Omega_2} &= - H_\Omega^{-1} e_{\partial \Omega_2}^\top H_{\partial \Omega_2} d_{\partial \Omega_2} , \\
		SAT^{(A)}_{\partial \Omega_3} &= - H_\Omega^{-1} e_{\partial \Omega_3}^\top H_{\partial \Omega_3} d_{\partial \Omega_3} , \\
		SAT^{(B)}_{\partial \Omega_3} &= - H_\Omega^{-1} e_{\partial \Omega_3}^\top H_{\partial \Omega_3} e_{\partial \Omega_3}, 
	\end{alignedat}
\end{equation}
and the projection operator by
\begin{equation}
	P = I_{\hat N} - H_\Omega^{-1} L^\top (L H_\Omega^{-1} L^\top)^{-1} L,
\end{equation}
where $I_{\hat N}$ denotes the $\hat N \times \hat N$ identity matrix and
\begin{equation}
	\label{eq: disc_wave_bc_L}
	L = e_{\partial \Omega_1}.
\end{equation}
The boundary operators in \eqref{eq: disc_wave_bc_sats} and \eqref{eq: disc_wave_bc_L} are given in \eqref{eq: disc_example_boundary_ops}. The discretization of the Dirac delta function $\delta(x-x_s,y-y_s)$ is done by ensuring that $(x_s,y_s)$ lies in a grid point and setting
\begin{equation}
	(d_s)_i = 
	\begin{cases}
		(H^{-1}_\Omega)_{i,i}, \quad &\text{if } x_i = x_s \text{ and } y_i = y_s, \\
		0, &\text{otherwise},
	\end{cases}
\end{equation}
where $(H^{-1}_\Omega)_{i,i}$ is the $i$:th diagonal entry of $H^{-1}_\Omega$. There exist discretizations for arbitrarily located point sources, see for example \cite{PETERSSON2016532}, but this is not the focus of the present work and is considered out of scope.
\begin{lemma}
	The ODE system \eqref{eq: disc_wave_eq} is stable.
\end{lemma}
\begin{proof}
	I prove stability using the energy method, for which it is sufficient to consider the homogeneous problem, i.e., $\bv{F}(t) = 0$. Taking the inner product $(\cdot,\cdot)_H$ between $\bv{v}_t$ and \eqref{eq: disc_wave_eq} results in
\begin{equation}
	(\bv{v}_t,\bv{v}_{tt})_{H_\Omega} = (\bv{v}_t,A v + B \bv{v}_t)_{H_\Omega}.
\end{equation}
Using the definitions of $A$ and $B$ results in
\begin{equation}
	\begin{alignedat}{2}
		(\bv{v}_t,\bv{v}_{tt})_{H_\Omega} &= - c^2 (D^+_x E \tilde{\bv{v}}_t, D^+_x E \tilde{\bv{v}})_{H_\Omega^+} - c^2 (D^+_y E \tilde{\bv{v}}_t, D^+_y E \tilde{\bv{v}})_{H_\Omega^+} \\
		&+ c^2 \langle e_{\partial \Omega_1} \tilde{\bv{v}}_t, d_{\partial \Omega_1} \tilde{\bv{v}} \rangle_{H_{\partial \Omega_1}} - c \snorm{e_{\partial \Omega_3} \tilde{\bv{v}}_t }_{H_{\partial \Omega_3}},
	\end{alignedat}
\end{equation}
where $\tilde{\bv{v}} = P \bv{v}$. Using that $L\tilde{\bv{v}} = L P \bv{v} = 0$ and hence $e_{\partial \Omega_1} \tilde{\bv{v}}_t = 0$ (by the definition of $P$), and rearranging terms leads to the semi-discrete energy estimate
\begin{equation}
	\label{eq: disc_wave_eneq}
 	\frac{d}{dt} (\snorm{\bv{v}_t}_{H_\Omega} + \snorm{D^+_x E \tilde{\bv{v}}}_{H_\Omega^+} + \snorm{D^+_y E \tilde{\bv{v}}}_{H_\Omega^+}) = - 2 \snorm{e_{\partial \Omega_3} \tilde{\bv{v}}_t }_{H_{\partial \Omega_3}} \leq 0,
 \end{equation} 
which proves stability.
\end{proof}
Note that \eqref{eq: disc_wave_eneq} is the semi-discrete equivalent to \eqref{eq: cont_wave_eneq}.
\section{Numerical results}
\label{sec: num_res}
Two numerical experiments are done to evaluate the new method. First, the accuracy and convergence properties of the method are compared against currently available high-order finite difference methods. In the final experiment, the method is applied to solve a more physically relevant problem, demonstrating a potential use case for the method in the future. 

For both problems, the ODE \eqref{eq: disc_wave_eq} is written on first-order form and solved using the classical fourth-order Runge-Kutta method. This is an accurate and explicit scheme with a stability region that includes a part of the imaginary axis, which makes it suitable for highly accurate discretizations of the wave equation. The time step is chosen as 1/5 of the stability limit, which is sufficient to ensure that the temporal error is insignificant compared to the spatial error.
\subsection{Accuracy}
Consider a circle of radius one centered at the origin. With a point source in $(x_s,y_s)$ and temporal component
\begin{equation}
	\label{eq: cont_forcing}
	f(t) = \frac{1}{\sigma \sqrt{2 \pi}} e^\frac{{-(t - t_s)^2}}{2 \sigma^2},
\end{equation}
the solution is given by
\begin{equation}
	\label{eq: analytical_u}
 	u(x,y,t) = \frac{1}{\sigma (2 \pi)^{3/2}} \int_0^\infty e^\frac{{-(t - t_s - r(x,y) \cosh(\omega))^2}}{2 \sigma^2} \: d\omega,
\end{equation} 
where $r(x,y) = \sqrt{(x - x_s)^2 + (y - y_s)^2}$ \cite{STIERNSTROM2023112376}. Here I use $\sigma = 0.04$ and $t_s = 0.3$. Since the integrand in \eqref{eq: analytical_u} rapidly goes to zero as $\omega \rightarrow \infty$, it can be easily computed numerically up to the required tolerance. Here I use the MATLAB function \verb|integral|.

The domain is decomposed into blocks as depicted in Figure \ref{fig: circle_domain} for the traditional and boundary-optimized finite difference operators and into quadrilaterals using HOHQMesh \cite{hohqmesh} for the SBP GL operators. In Figure \ref{fig: circle_domain_hohq} the decomposition into quadrilaterals with 340 elements is shown.
\begin{figure}
\label{fig: domain_circle_both}
\begin{subfigure}{.48\textwidth}
		\center
		\includegraphics[width=0.95\textwidth]{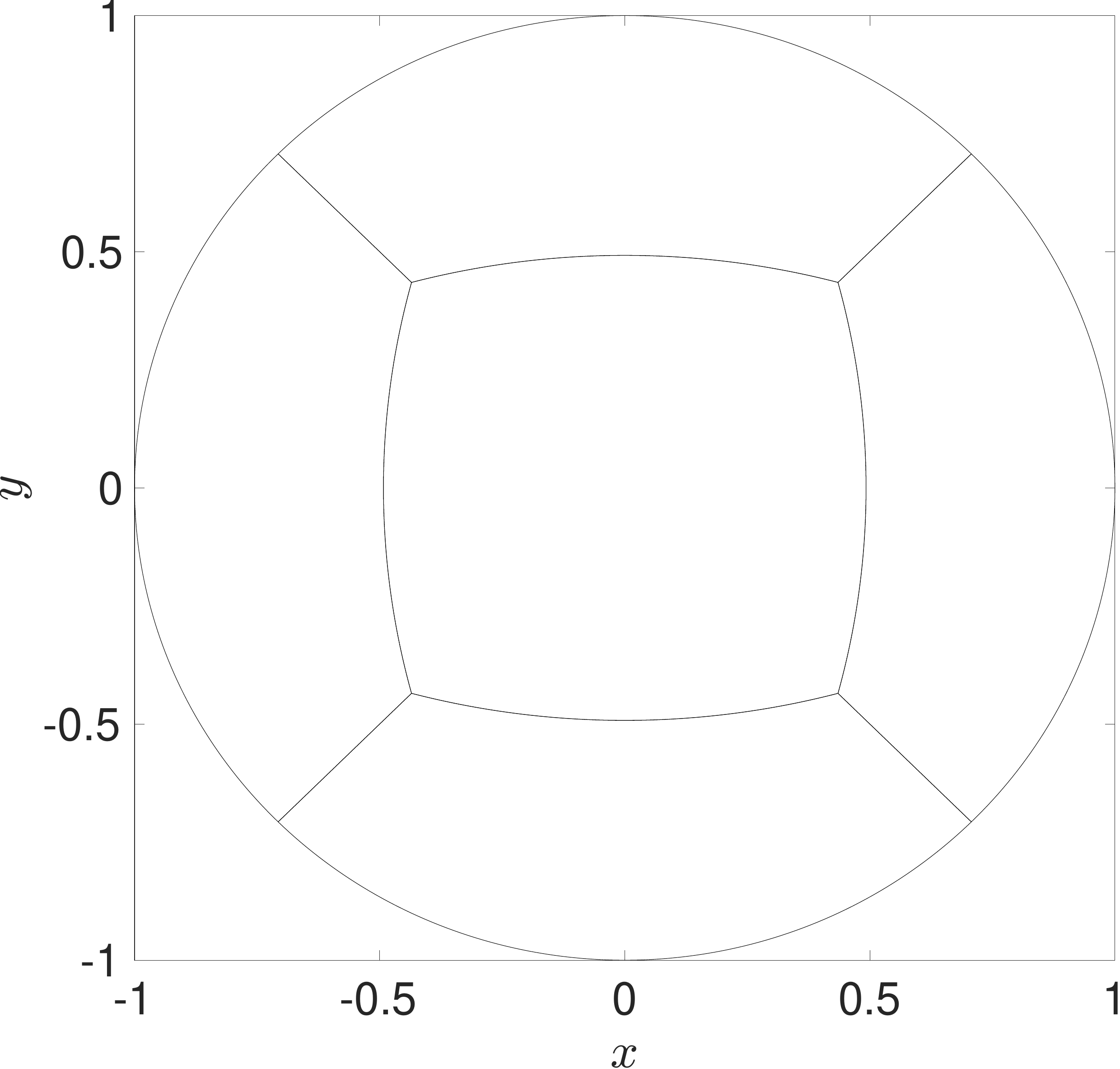}
		\caption{Domain decomposition with traditional FD operators.}
		\label{fig: circle_domain}
\end{subfigure}
\begin{subfigure}{.48\textwidth}
		\center
		\includegraphics[width=0.95\textwidth]{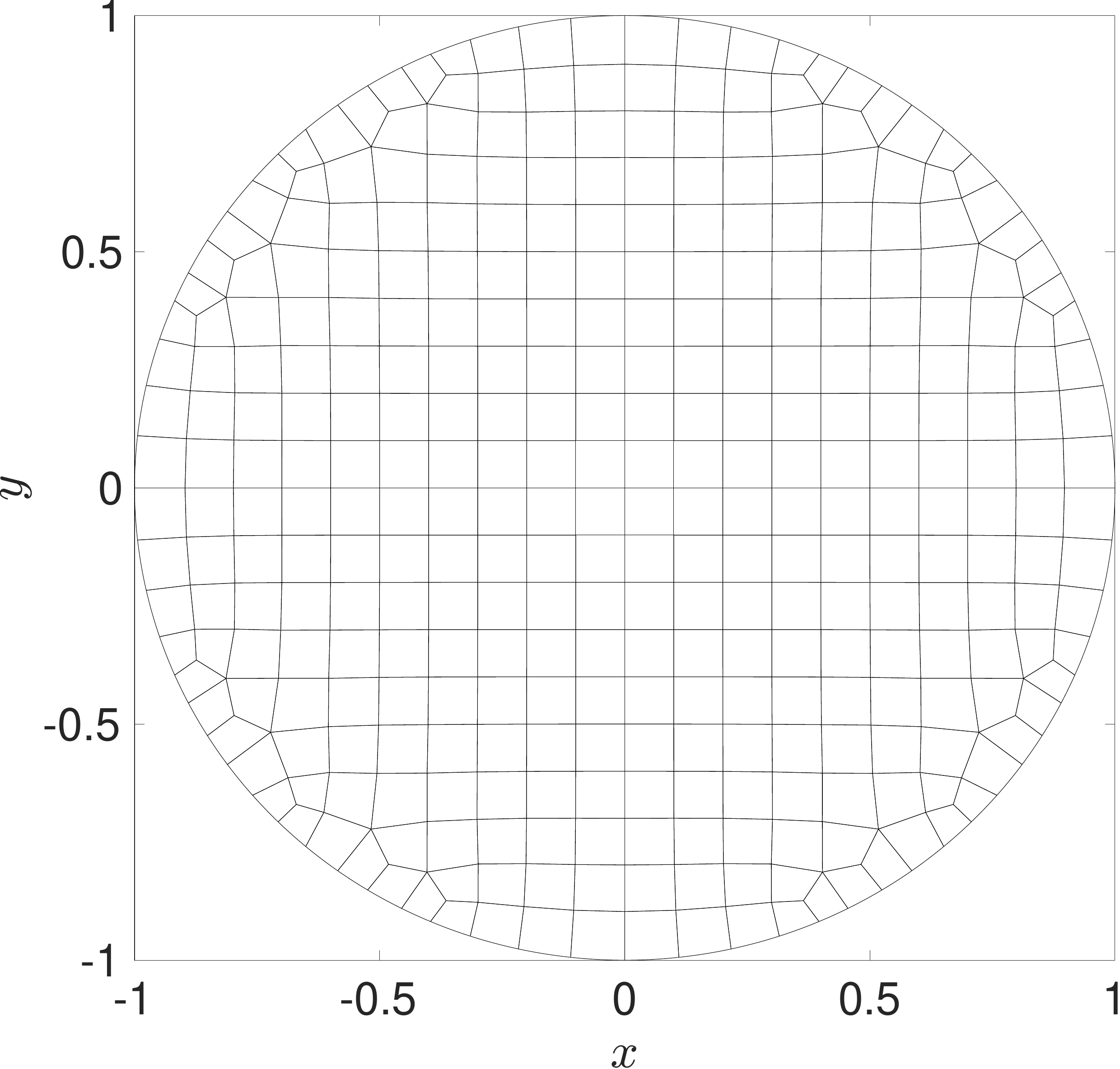}
		\caption{Domain decomposition with SBP GL operators.}
		\label{fig: circle_domain_hohq}
\end{subfigure}
\caption{Domain decompositions of the unit circle.}
\end{figure}

In Figure \ref{fig: conv_plot} the $L_2$-error at $t = 0.8$ of the 5th, 7th, and 9th order SBP GL operators on meshes with varying number of quadrilaterals (generated using HOHQMesh) is plotted against the number of DOFs. In the same plot, the $L_2$-error is plotted against DOFs with the 2nd, 4th, and 6th order traditional operators (equidistant) and the 8th, 10th, and 12th order boundary optimized operators on the block decomposition in Figure \ref{fig: circle_domain}, with interface conditions imposed using the hybrid projection and SAT method \cite{eriksson2023acoustic}. In Figure \ref{fig: eff_plot} the $L_2$-error versus runtime is plotted for the SBP GL, traditional, and boundary-optimized operators. The runtime is measured using the MATLAB commands \verb|tic| and \verb|toc| around the time stepping loop, i.e., the setup costs are not included, on a 10-core Intel Xeon E5-2630 CPU. The error results with the SBP GL operators are also presented in Table \ref{tabl: errconv_GL}, including convergence rates estimated as
\begin{equation}
	q = \frac{\log{\left(\frac{e_1}{e_2}\right)}}{\log{\left(\frac{N_2}{N_1}\right)^{1/2}}},
\end{equation}
where $e_1$ and $e_2$ are the $L_2$-errors with $N_1$ and $N_2$ DOFs.

From the results, it is clear that the SBP GL operators used together with the continuous SBP method leads to a scheme that is comparable to the current state of the art with traditional and boundary-optimized finite difference operators in terms of error for a given runtime, and even more efficient for very accurate solutions. Additionally, the results show that the estimated convergence rate with a $q$:th order SBP GL operator is at least $q+2$. 

\begin{figure}
\label{fig: conv_eff_plot}
\begin{subfigure}{.48\textwidth}
		\center
		\includegraphics[width=0.95\textwidth]{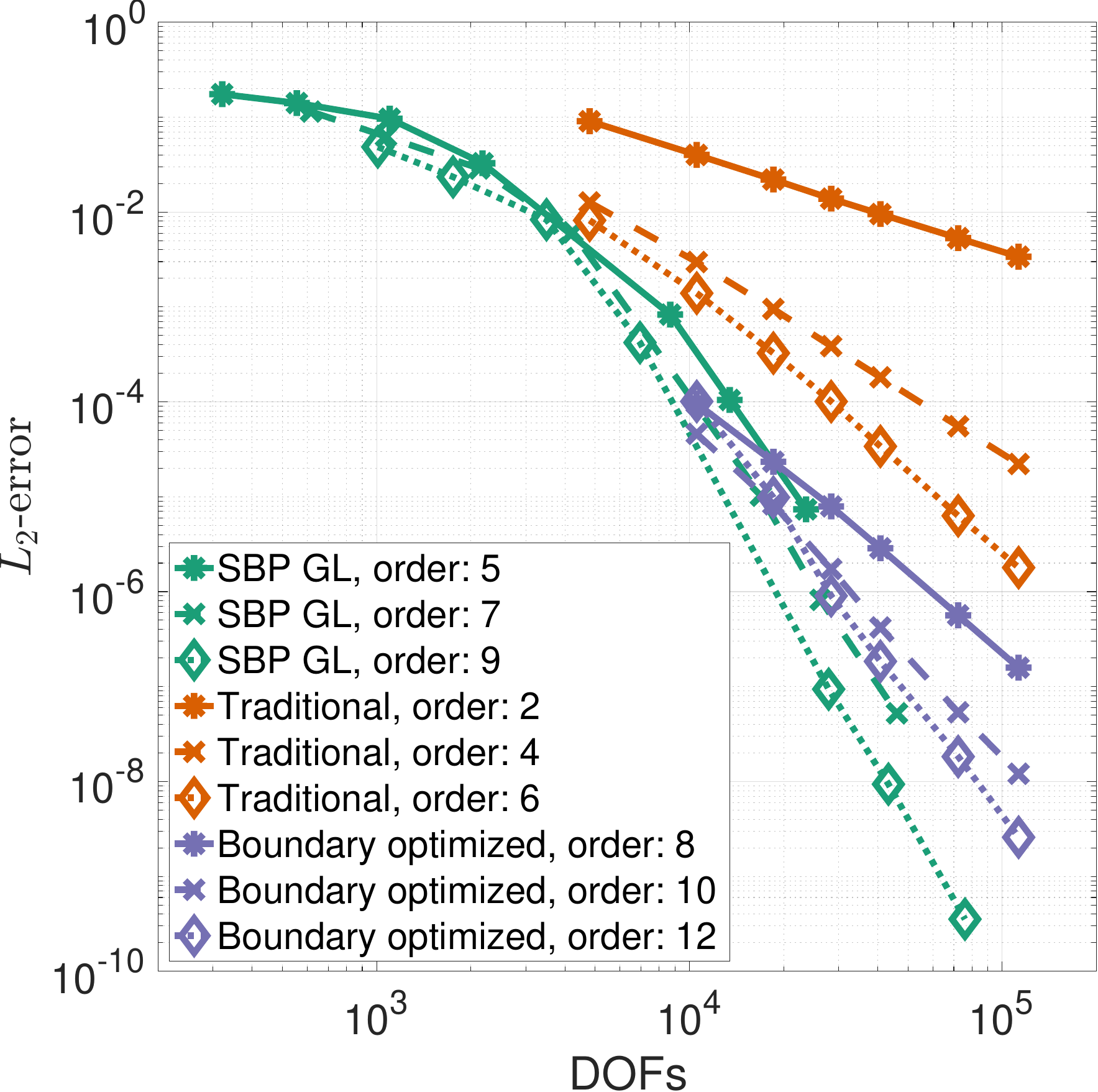}
		\caption{$L_2$-error versus degrees of freedom.}
		\label{fig: conv_plot}
\end{subfigure}
\begin{subfigure}{.48\textwidth}
		\center
		\includegraphics[width=0.95\textwidth]{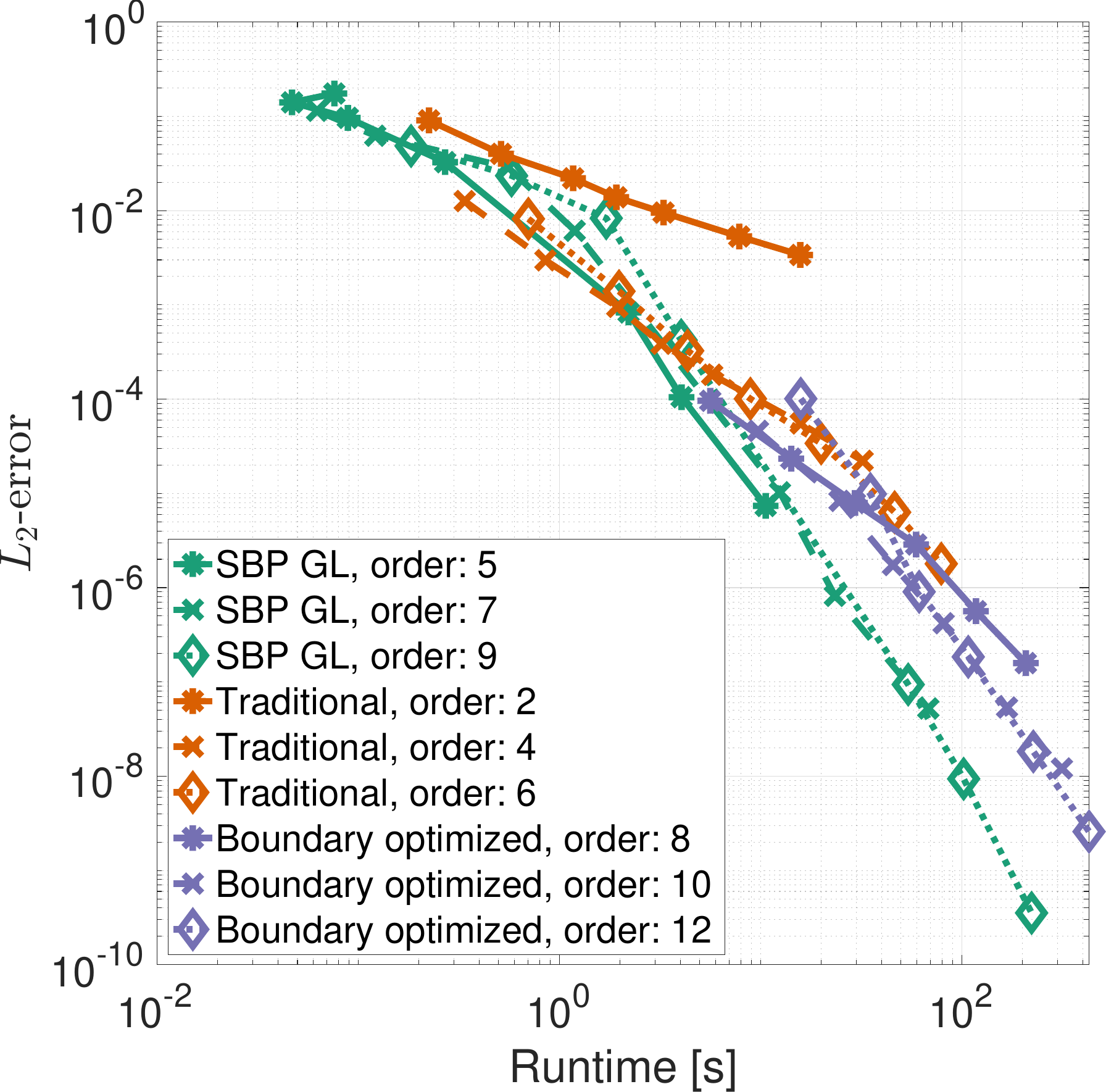}
		\caption{$L_2$-error versus runtime.}
		\label{fig: eff_plot}
\end{subfigure}
\caption{$L_2$-error versus degrees of freedom and runtime for the continuous SBP-SAT method with SBP GL operators, the SBP-P-SAT method with traditional SBP operators, and the SBP-P-SAT method with boundary optimized operators.}
\end{figure}

\begin{table}
	\centering
	\caption{$L_2$-errors (in base 10 logarithms) and convergence rates for the continuous SBP-SAT method with GL operators for varying degrees of freedom. The number of elements (blocks) is presented in the leftmost column.}
	\label{tabl: errconv_GL}
	\begin{tabular}{|c||c|c|c||c|c|c||c|c|c|} 
	\hline
	$N_{ele}$ & $N_{dofs}^{(5)}$ & $e^{(5)}$ & $q^{(5)}$ & $N_{dofs}^{(7)}$ & $e^{(7)}$ & $q^{(7)}$ & $N_{dofs}^{(9)}$ & $e^{(9)}$ & $q^{(9)}$ \\
	\hline
	12 & 321 	& -0.76 & - 	& 617 	& -0.94 & - 		& 1009 	& -1.31 & - \\
	21 & 556  	& -0.85 & 0.79 & 1072 	& -1.21 & 2.20 	& 1756 	& -1.63 & 2.62 	\\
	42 & 1101 	& -1.02 & 1.11 & 2129 	& -1.54 & 2.25 	& 3493 	& -2.08 & 3.02 	\\
	84 & 2181 	& -1.49 & 3.16 & 4229 	& -2.21 & 4.51 	& 6949 	& -3.38 & 8.67 	\\
	340 & 8681 	& -3.08 & 5.32 & 16913 & -4.99 & 9.22 	& 27865 & -7.03 & 12.11 \\
	529 & 13456 & -3.98 & 9.45 & 26244 & -6.09 & 11.51 	& 43264 & -8.03 & 10.48 \\
	931 & 23596 & -5.13 & 9.45 & 46068 & -7.28 & 9.78 	& 75988 & -9.45 & 11.63 \\
	\hline
	\end{tabular}
\end{table}
\subsection{Acoustic simulation}
To illustrate the method's potential uses, a numerical experiment with a complex geometry is done. Consider the domain in Figure \ref{fig: mb_domain}, split into two parts where the bottom part contains some complex shapes and the top part is rectangular. Dirichlet boundary conditions are imposed on $\partial \Omega_1$, Neumann boundary conditions on $\partial \Omega_2$, and first-order outflow boundary conditions on $\partial \Omega_3$. A point source is placed at $(50,50)$ with temporal component given by \eqref{eq: cont_forcing} with $\sigma = 0.005$ and $t_s = 0.03$. This problem setup could for example be used to model acoustic pressure waves propagating from a source in the water down into the seabed where solid objects of complex shapes are located. The homogeneous Dirichlet boundary condition at $\partial \Omega_1$ corresponds to constant pressure at the surface, the homogeneous Neumann boundary condition at $\partial \Omega_2$ to fully reflecting objects (rocks for example), and the outflow boundary conditions on $\partial \Omega_3$ to a truncation of the computational domain. The wave speeds are chosen to be 1500 m/s in the top block and 1800 m/s in the bottom block, corresponding roughly to the acoustic wave speeds in salt water and soil, respectively. The top part of the domain is discretized using the 6th-order traditional equidistant finite difference operators (one block) and the bottom part is decomposed into quadrilaterals (940 blocks) and discretized using the 5th-order SBP GL operators together with the continuous SBP method as described in Section \ref{sec: cont_SBP_lapl}. The block decomposition is done using HOHQMesh \cite{hohqmesh}.
\begin{figure}[h]
\begin{subfigure}{.48\textwidth}
		\center
		\includegraphics[width=0.95\textwidth]{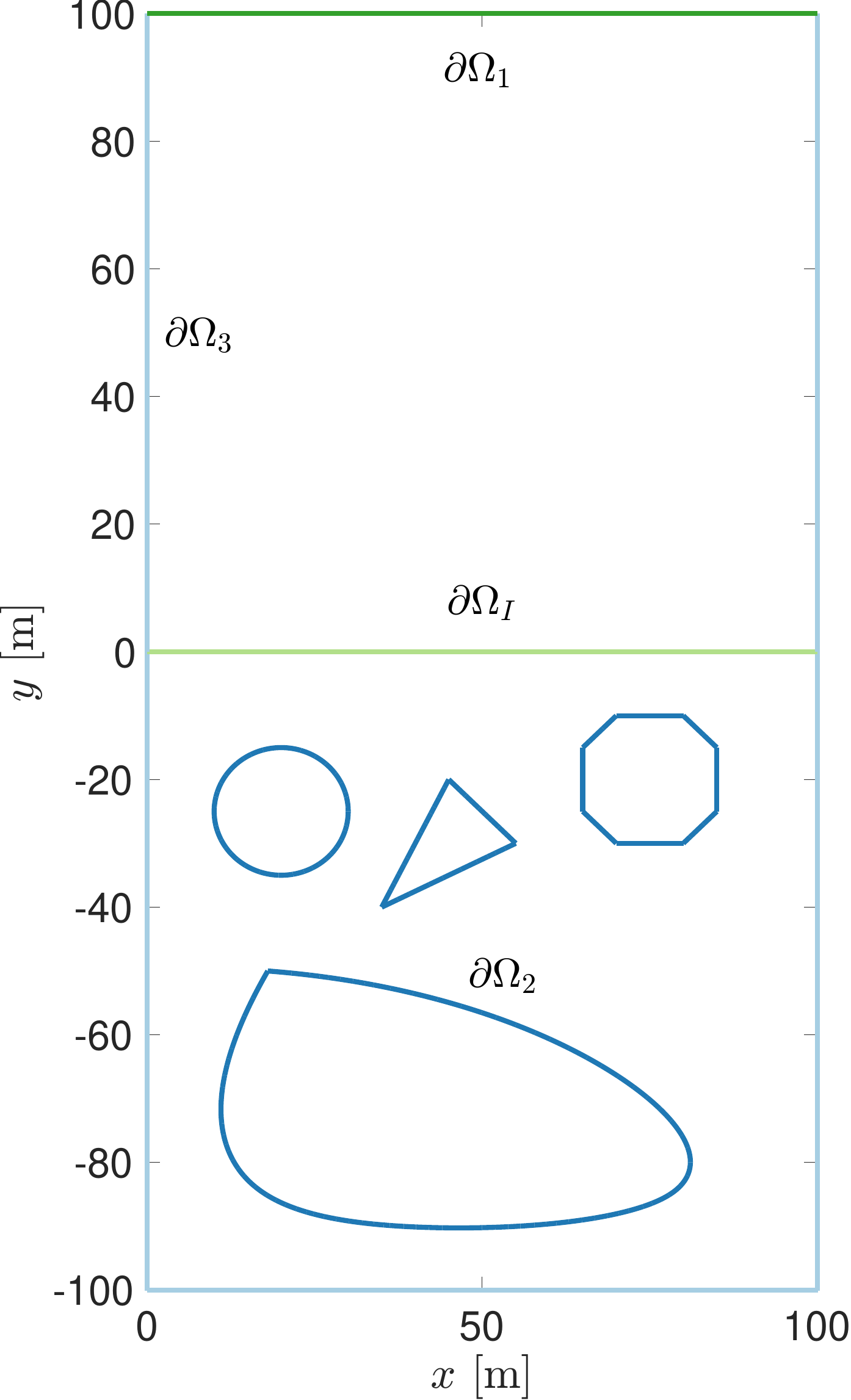}
\end{subfigure}
\begin{subfigure}{.48\textwidth}
		\center
		\includegraphics[width=0.95\textwidth]{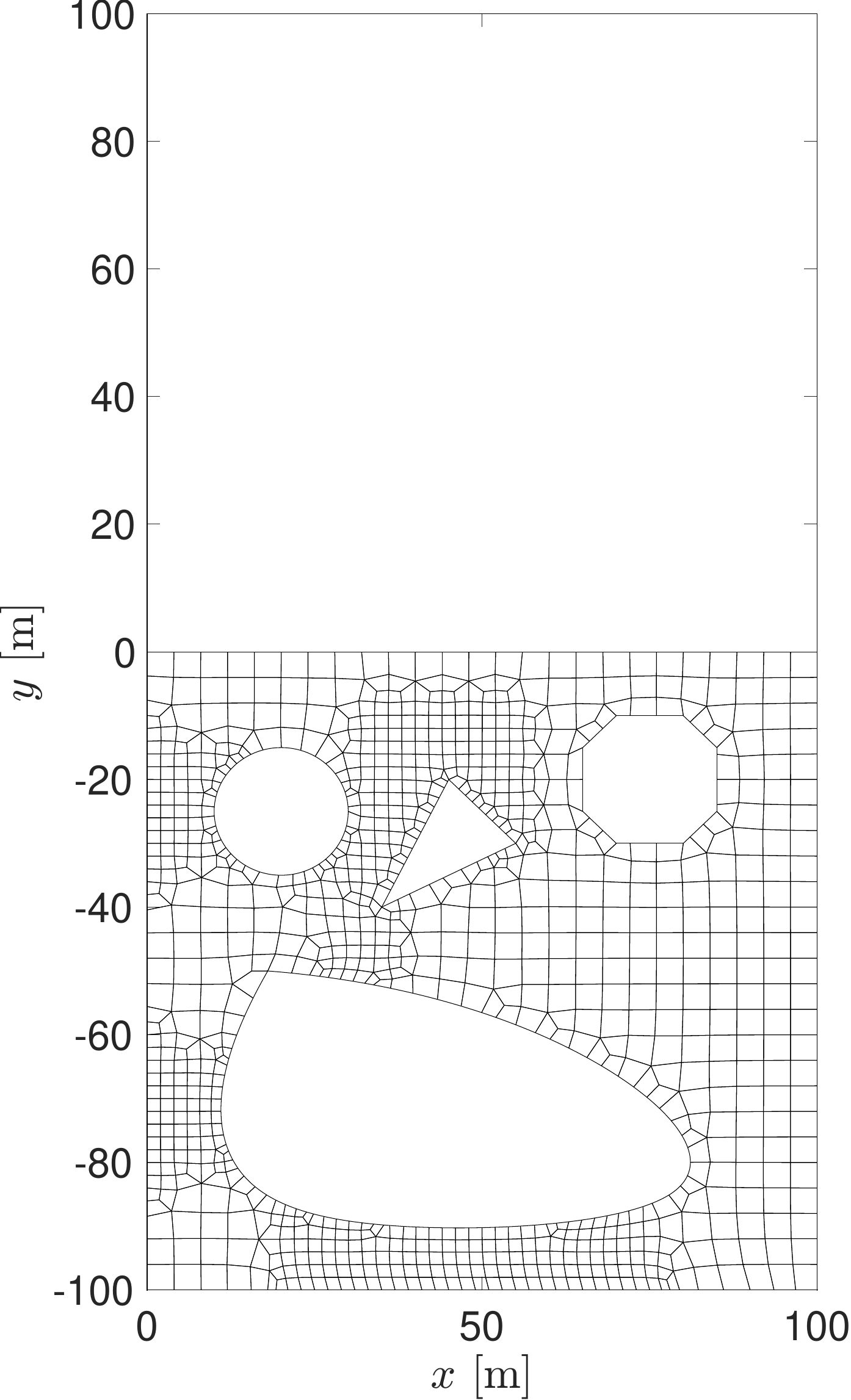}
\end{subfigure}
\caption{Boundaries (left figure) and block decomposition (right figure) of domain with complex shapes. The boundaries are color-coded in the left figure as follows: the interface $\partial \Omega_I$ is light green, the Dirichlet boundary $\partial \Omega_1$ is dark green, the Neumann boundaries $\partial \Omega_2$ are dark blue, and the outflow boundaries $\partial \Omega_3$ are light blue.}
\label{fig: mb_domain}
\end{figure}

The two parts of the domain are coupled by interface conditions at the interface $\partial \Omega_I$. Since the grid points at each side of the interface are not aligned (equidistant in the upper part and Gauss-Lobatto points in each element in the lower part), interpolation operators have to be used to impose the interface conditions. Here the glue-grid framework \cite{Kozdon2016} is used, where a series of projection operators are constructed and applied sequentially to interpolate from one side of the interface to the other. To impose the interface conditions a hybrid projection and SAT method is used, precisely as presented in \cite{Eriksson2023}. The focus of the current work is not on constructing interpolation operators or imposing non-conforming interface conditions, therefore more details on the implementation are not presented here. MATLAB code to reproduce the results is available at \url{https://github.com/guer7/sbp_embed}. Snapshots of the solution at $t = 50$ ms, $t = 75$ ms, $t = 100$ ms, and $t = 200$ ms are shown in Figure \ref{fig: mb_time_results}.
\begin{figure}
	\center
	\includegraphics[width=0.9\textwidth]{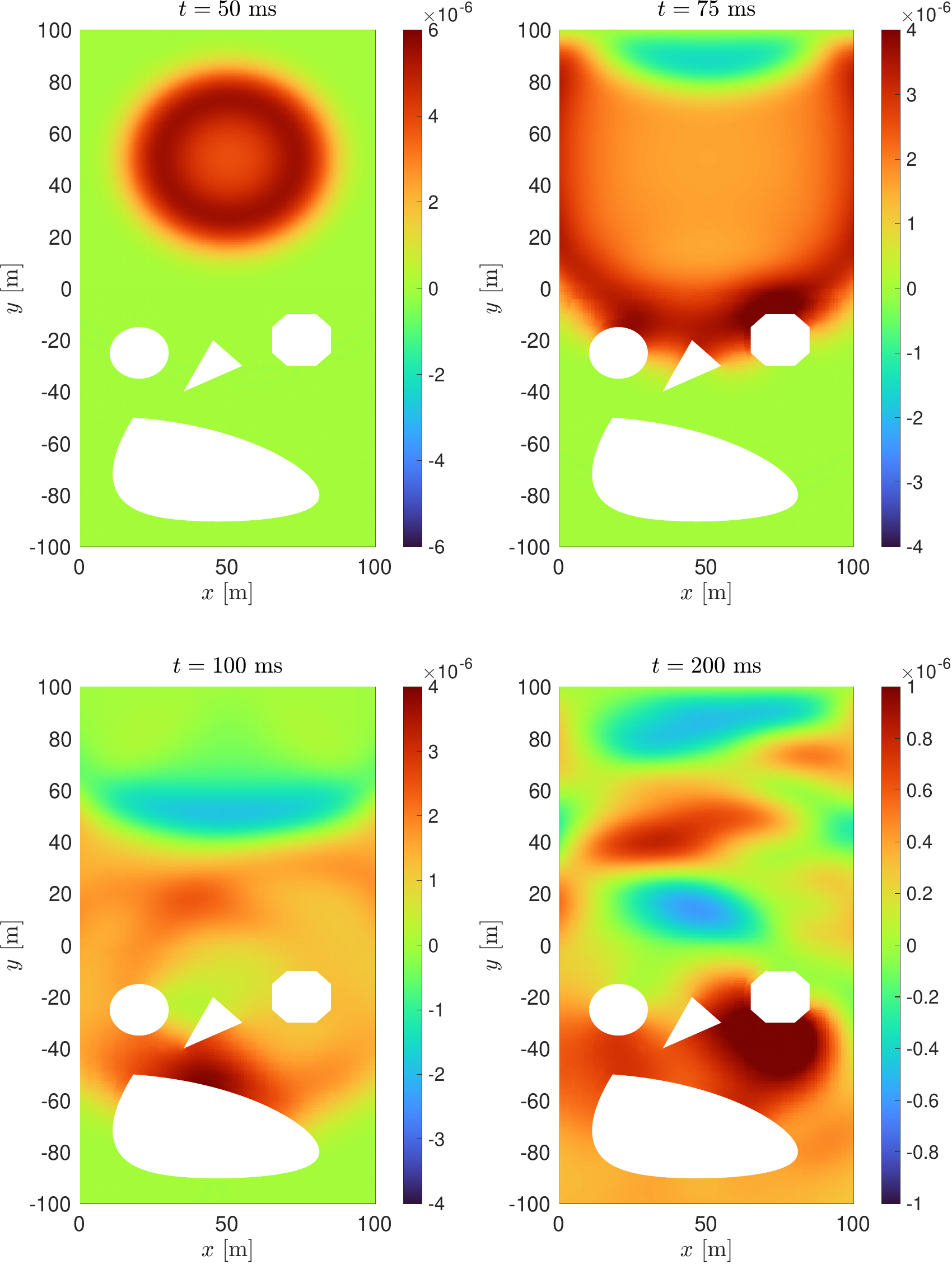}
	\caption{Numerical simulation results at $t = 50$ ms, $t = 75$ ms, $t = 100$ ms, and $t = 200$ ms of acoustic pressure waves emitting from a source in the upper block (water) and propagating into the lower block (soil) and reflecting off of the complex objects.}
	\label{fig: mb_time_results}
\end{figure}
\section{Conclusions}
\label{sec: concl}
The continuous summation-by-parts (SBP) method is used to derive an efficient discrete multiblock Laplace operator suitable for discretizations of domains with many small blocks. The interface conditions are imposed using a hybrid method where continuity of the solutions is imposed inherently in the scheme (using the continuous SBP approach) and continuity of the normal derivatives using a weak penalty method. The new method is paired together with SBP operators defined on Gauss-Lobatto (GL) quadrature points, which leads to a very efficient scheme (accurate with respect to computational costs). Since the continuous SBP method removes duplicated degrees of freedom at the interfaces and since the SBP GL operators are highly accurate for a small number of grid points, the new method is ideally suited for discretizations of complicated geometries that require a domain decomposition into many small blocks. In this work, quadrilateral blocks mapped to squares are used, but the method could also be applied to operators defined on simplex elements such as the ones presented in \cite{doi:10.1137/15M1038360}, for example.

Numerical experiments demonstrate that the new method is comparable in terms of runtime for a given $L_2$-error to the current state-of-the-art high-order accurate finite difference methods, and even more efficient for very accurate solutions. Furthermore, a numerical experiment on a more realistic problem with a complex geometry is done, where the new method is coupled to a traditional equidistant finite difference method using interpolation operators. This demonstrates a potential use case of the method where complex parts of the computational domain are discretized using the new method and the far-field regions using a traditional finite difference method. Although these results show that the new method is comparable in efficiency to traditional methods, the question remains whether this is still the case for larger problems in 3D where traditional finite differences are known to scale very well on modern hardware. Additional future work includes extensions to other equations, for example the Navier-Stokes equations and the dynamic Kirchhoff-Love plate equation, where highly accurate simulations around complex objects are of interest as well.

\section*{Statements and Declarations}
\noindent\textbf{Funding } The author was partially supported by the Swedish Research Council (grant 2021-05830 VR) and FORMAS (grant 2022-00843). 
\\
\noindent\textbf{Conflict of interest } The author has no conflicts of interest to declare that are relevant to the content of this article.

\noindent\textbf{Data availability } Data sharing not applicable to this article as no datasets were generated or analyzed during the current study.

\clearpage
\bibliography{references}


\begin{thebibliography}{44}
\ifx \bisbn   \undefined \def \bisbn  #1{ISBN #1}\fi
\ifx \binits  \undefined \def \binits#1{#1}\fi
\ifx \bauthor  \undefined \def \bauthor#1{#1}\fi
\ifx \batitle  \undefined \def \batitle#1{#1}\fi
\ifx \bjtitle  \undefined \def \bjtitle#1{#1}\fi
\ifx \bvolume  \undefined \def \bvolume#1{\textbf{#1}}\fi
\ifx \byear  \undefined \def \byear#1{#1}\fi
\ifx \bissue  \undefined \def \bissue#1{#1}\fi
\ifx \bfpage  \undefined \def \bfpage#1{#1}\fi
\ifx \blpage  \undefined \def \blpage #1{#1}\fi
\ifx \burl  \undefined \def \burl#1{\textsf{#1}}\fi
\ifx \doiurl  \undefined \def \doiurl#1{\url{https://doi.org/#1}}\fi
\ifx \betal  \undefined \def \betal{\textit{et al.}}\fi
\ifx \binstitute  \undefined \def \binstitute#1{#1}\fi
\ifx \binstitutionaled  \undefined \def \binstitutionaled#1{#1}\fi
\ifx \bctitle  \undefined \def \bctitle#1{#1}\fi
\ifx \beditor  \undefined \def \beditor#1{#1}\fi
\ifx \bpublisher  \undefined \def \bpublisher#1{#1}\fi
\ifx \bbtitle  \undefined \def \bbtitle#1{#1}\fi
\ifx \bedition  \undefined \def \bedition#1{#1}\fi
\ifx \bseriesno  \undefined \def \bseriesno#1{#1}\fi
\ifx \blocation  \undefined \def \blocation#1{#1}\fi
\ifx \bsertitle  \undefined \def \bsertitle#1{#1}\fi
\ifx \bsnm \undefined \def \bsnm#1{#1}\fi
\ifx \bsuffix \undefined \def \bsuffix#1{#1}\fi
\ifx \bparticle \undefined \def \bparticle#1{#1}\fi
\ifx \barticle \undefined \def \barticle#1{#1}\fi
\bibcommenthead
\ifx \bconfdate \undefined \def \bconfdate #1{#1}\fi
\ifx \botherref \undefined \def \botherref #1{#1}\fi
\ifx \url \undefined \def \url#1{\textsf{#1}}\fi
\ifx \bchapter \undefined \def \bchapter#1{#1}\fi
\ifx \bbook \undefined \def \bbook#1{#1}\fi
\ifx \bcomment \undefined \def \bcomment#1{#1}\fi
\ifx \oauthor \undefined \def \oauthor#1{#1}\fi
\ifx \citeauthoryear \undefined \def \citeauthoryear#1{#1}\fi
\ifx \endbibitem  \undefined \def \endbibitem {}\fi
\ifx \bconflocation  \undefined \def \bconflocation#1{#1}\fi
\ifx \arxivurl  \undefined \def \arxivurl#1{\textsf{#1}}\fi
\csname PreBibitemsHook\endcsname

\bibitem{Kreiss1972}
\begin{barticle}
\bauthor{\bsnm{Kreiss}, \binits{H.-O.}},
\bauthor{\bsnm{Oliger}, \binits{J.}}:
\batitle{Comparison of accurate methods for the integration of hyperbolic
  equations}.
\bjtitle{Tellus}
\bvolume{24},
\bfpage{199}--\blpage{215}
(\byear{1972}).
\doiurl{10.3402/tellusa.v24i3.10634}
\end{barticle}
\endbibitem

\bibitem{doi:10.1137/0711076}
\begin{barticle}
\bauthor{\bsnm{Swartz}, \binits{B.}},
\bauthor{\bsnm{Wendroff}, \binits{B.}}:
\batitle{The relative efficiency of finite difference and finite element
  methods. {I}: Hyperbolic problems and splines}.
\bjtitle{SIAM Journal on Numerical Analysis}
\bvolume{11}(\bissue{5}),
\bfpage{979}--\blpage{993}
(\byear{1974}).
\doiurl{10.1137/0711076}
\end{barticle}
\endbibitem

\bibitem{NORDSTROM2003453}
\begin{barticle}
\bauthor{\bsnm{Nordström}, \binits{J.}},
\bauthor{\bsnm{Forsberg}, \binits{K.}},
\bauthor{\bsnm{Adamsson}, \binits{C.}},
\bauthor{\bsnm{Eliasson}, \binits{P.}}:
\batitle{Finite volume methods, unstructured meshes and strict stability
  for hyperbolic problems}.
\bjtitle{Applied Numerical Mathematics}
\bvolume{45}(\bissue{4}),
\bfpage{453}--\blpage{473}
(\byear{2003}).
\doiurl{10.1016/S0168-9274(02)00239-8}
\end{barticle}
\endbibitem

\bibitem{NORDSTROM2009875}
\begin{barticle}
\bauthor{\bsnm{Nordström}, \binits{J.}},
\bauthor{\bsnm{Ham}, \binits{F.}},
\bauthor{\bsnm{Shoeybi}, \binits{M.}},
\bauthor{\bsnm{{van der Weide}}, \binits{E.}},
\bauthor{\bsnm{Svärd}, \binits{M.}},
\bauthor{\bsnm{Mattsson}, \binits{K.}},
\bauthor{\bsnm{Iaccarino}, \binits{G.}},
\bauthor{\bsnm{Gong}, \binits{J.}}:
\batitle{A hybrid method for unsteady inviscid fluid flow}.
\bjtitle{Computers \& Fluids}
\bvolume{38}(\bissue{4}),
\bfpage{875}--\blpage{882}
(\byear{2009}).
\doiurl{10.1016/j.compfluid.2008.09.010}
\end{barticle}
\endbibitem

\bibitem{doi:10.1137/130932193}
\begin{barticle}
\bauthor{\bsnm{Carpenter}, \binits{M.H.}},
\bauthor{\bsnm{Fisher}, \binits{T.C.}},
\bauthor{\bsnm{Nielsen}, \binits{E.J.}},
\bauthor{\bsnm{Frankel}, \binits{S.H.}}:
\batitle{Entropy stable spectral collocation schemes for the {N}avier--{S}tokes
  equations: Discontinuous interfaces}.
\bjtitle{SIAM Journal on Scientific Computing}
\bvolume{36}(\bissue{5}),
\bfpage{835}--\blpage{867}
(\byear{2014}).
\doiurl{10.1137/130932193}
\end{barticle}
\endbibitem

\bibitem{CHAN2018346}
\begin{barticle}
\bauthor{\bsnm{Chan}, \binits{J.}}:
\batitle{On discretely entropy conservative and entropy stable discontinuous
  {G}alerkin methods}.
\bjtitle{Journal of Computational Physics}
\bvolume{362},
\bfpage{346}--\blpage{374}
(\byear{2018}).
\doiurl{10.1016/j.jcp.2018.02.033}
\end{barticle}
\endbibitem

\bibitem{https://doi.org/10.1002/fld.3923}
\begin{barticle}
\bauthor{\bsnm{Gassner}, \binits{G.J.}}:
\batitle{A kinetic energy preserving nodal discontinuous {G}alerkin spectral
  element method}.
\bjtitle{International Journal for Numerical Methods in Fluids}
\bvolume{76}(\bissue{1}),
\bfpage{28}--\blpage{50}
(\byear{2014}).
\doiurl{10.1002/fld.3923}
\end{barticle}
\endbibitem

\bibitem{Carpenter1994}
\begin{barticle}
\bauthor{\bsnm{Carpenter}, \binits{M.H.}},
\bauthor{\bsnm{Gottlieb}, \binits{D.}},
\bauthor{\bsnm{Abarbanel}, \binits{S.}}:
\batitle{Time-stable boundary conditions for finite-difference schemes solving
  hyperbolic systems: methodology and application to high-order compact
  schemes}.
\bjtitle{J. Comput. Phys.}
\bvolume{111}(\bissue{2}),
\bfpage{220}--\blpage{236}
(\byear{1994}).
\doiurl{10.1006/jcph.1994.1057}
\end{barticle}
\endbibitem

\bibitem{DelReyFernandez2014}
\begin{botherref}
\oauthor{\bsnm{Del Rey~Fern\'{a}ndez}, \binits{D.C.}},
\oauthor{\bsnm{Hicken}, \binits{J.E.}},
\oauthor{\bsnm{Zingg}, \binits{D.W.}}:
Review of summation-by-parts operators with simultaneous approximation terms
  for the numerical solution of partial differential equations
(2014).
\doiurl{10.1016/j.compfluid.2014.02.016}
\end{botherref}
\endbibitem

\bibitem{Svard2014}
\begin{barticle}
\bauthor{\bsnm{Sv\"{a}rd}, \binits{M.}},
\bauthor{\bsnm{Nordstr\"{o}m}, \binits{J.}}:
\batitle{Review of summation-by-parts schemes for initial-boundary-value
  problems}.
\bjtitle{J. Comput. Phys.}
\bvolume{268},
\bfpage{17}--\blpage{38}
(\byear{2014}).
\doiurl{10.1016/j.jcp.2014.02.031}
\end{barticle}
\endbibitem

\bibitem{Olsson1995a}
\begin{barticle}
\bauthor{\bsnm{Olsson}, \binits{P.}}:
\batitle{Summation by parts, projections, and stability. {I}}.
\bjtitle{Math. Comp.}
\bvolume{64}(\bissue{211}),
\bfpage{1035}--\blpage{10652326}
(\byear{1995}).
\doiurl{10.2307/2153482}
\end{barticle}
\endbibitem

\bibitem{Olsson1995}
\begin{barticle}
\bauthor{\bsnm{Olsson}, \binits{P.}}:
\batitle{Summation by parts, projections, and stability. {II}}.
\bjtitle{Math. Comp.}
\bvolume{64}(\bissue{212}),
\bfpage{1473}--\blpage{1493}
(\byear{1995}).
\doiurl{10.2307/2153366}
\end{barticle}
\endbibitem

\bibitem{Mattsson2018}
\begin{barticle}
\bauthor{\bsnm{Mattsson}, \binits{K.}},
\bauthor{\bsnm{Almquist}, \binits{M.}},
\bauthor{\bparticle{van~der} \bsnm{Weide}, \binits{E.}}:
\batitle{Boundary optimized diagonal-norm {SBP} operators}.
\bjtitle{J. Comput. Phys.}
\bvolume{374},
\bfpage{1261}--\blpage{1266}
(\byear{2018}).
\doiurl{10.1016/j.jcp.2018.06.010}
\end{barticle}
\endbibitem

\bibitem{CiCP-8-1074}
\begin{barticle}
\bauthor{\bsnm{Petersson}, \binits{N.A.}},
\bauthor{\bsnm{Sj\"{o}green}, \binits{B.}}:
\batitle{Stable grid refinement and singular source discretization for seismic
  wave simulations}.
\bjtitle{Commun. Comput. Phys.}
\bvolume{8}(\bissue{5}),
\bfpage{1074}--\blpage{1110}
(\byear{2010}).
\doiurl{10.4208/cicp.041109.120210a}
\end{barticle}
\endbibitem

\bibitem{doi:10.1137/20M1339702}
\begin{barticle}
\bauthor{\bsnm{Zhang}, \binits{L.}},
\bauthor{\bsnm{Wang}, \binits{S.}},
\bauthor{\bsnm{Petersson}, \binits{N.A.}}:
\batitle{Elastic wave propagation in curvilinear coordinates with mesh
  refinement interfaces by a fourth order finite difference method}.
\bjtitle{SIAM Journal on Scientific Computing}
\bvolume{43}(\bissue{2}),
\bfpage{1472}--\blpage{1496}
(\byear{2021}).
\doiurl{10.1137/20M1339702}
\end{barticle}
\endbibitem

\bibitem{doi:10.1137/17M1139333}
\begin{barticle}
\bauthor{\bsnm{Linders}, \binits{V.}},
\bauthor{\bsnm{Lundquist}, \binits{T.}},
\bauthor{\bsnm{Nordstr\"{o}m}, \binits{J.}}:
\batitle{On the order of accuracy of finite difference operators on diagonal
  norm based summation-by-parts form}.
\bjtitle{SIAM Journal on Numerical Analysis}
\bvolume{56}(\bissue{2}),
\bfpage{1048}--\blpage{1063}
(\byear{2018}).
\doiurl{10.1137/17M1139333}
\end{barticle}
\endbibitem

\bibitem{MATTSSON201491}
\begin{barticle}
\bauthor{\bsnm{Mattsson}, \binits{K.}},
\bauthor{\bsnm{Almquist}, \binits{M.}},
\bauthor{\bsnm{Carpenter}, \binits{M.H.}}:
\batitle{Optimal diagonal-norm sbp operators}.
\bjtitle{Journal of Computational Physics}
\bvolume{264},
\bfpage{91}--\blpage{111}
(\byear{2014}).
\doiurl{10.1016/j.jcp.2013.12.041}
\end{barticle}
\endbibitem

\bibitem{STIERNSTROM2023112376}
\begin{barticle}
\bauthor{\bsnm{Stiernstr\"{o}m}, \binits{V.}},
\bauthor{\bsnm{Almquist}, \binits{M.}},
\bauthor{\bsnm{Mattsson}, \binits{K.}}:
\batitle{Boundary-optimized summation-by-parts operators for finite difference
  approximations of second derivatives with variable coefficients}.
\bjtitle{J. Comput. Phys.}
\bvolume{491},
\bfpage{112376}--\blpage{24}
(\byear{2023}).
\doiurl{10.1016/j.jcp.2023.112376}
\end{barticle}
\endbibitem

\bibitem{Mattsson2007}
\begin{barticle}
\bauthor{\bsnm{Mattsson}, \binits{K.}},
\bauthor{\bsnm{Sv{\"{a}}rd}, \binits{M.}},
\bauthor{\bsnm{Carpenter}, \binits{M.}},
\bauthor{\bsnm{Nordstr{\"{o}}m}, \binits{J.}}:
\batitle{{High-order accurate computations for unsteady aerodynamics}}.
\bjtitle{Computers and Fluids}
\bvolume{36}(\bissue{3}),
\bfpage{636}--\blpage{649}
(\byear{2007}).
\doiurl{10.1016/j.compfluid.2006.02.004}
\end{barticle}
\endbibitem

\bibitem{TAZHIMBETOV2023112470}
\begin{barticle}
\bauthor{\bsnm{Tazhimbetov}, \binits{N.}},
\bauthor{\bsnm{Almquist}, \binits{M.}},
\bauthor{\bsnm{Werpers}, \binits{J.}},
\bauthor{\bsnm{Dunham}, \binits{E.M.}}:
\batitle{Simulation of flexural-gravity wave propagation for elastic plates in
  shallow water using an energy-stable finite difference method with weakly
  enforced boundary and interface conditions}.
\bjtitle{Journal of Computational Physics}
\bvolume{493},
\bfpage{112470}
(\byear{2023}).
\doiurl{10.1016/j.jcp.2023.112470}
\end{barticle}
\endbibitem

\bibitem{Almquist2020}
\begin{barticle}
\bauthor{\bsnm{Almquist}, \binits{M.}},
\bauthor{\bsnm{Dunham}, \binits{E.M.}}:
\batitle{Non-stiff boundary and interface penalties for narrow-stencil finite
  difference approximations of the {L}aplacian on curvilinear multiblock
  grids}.
\bjtitle{J. Comput. Phys.}
\bvolume{408},
\bfpage{109294}--\blpage{33}
(\byear{2020}).
\doiurl{10.1016/j.jcp.2020.109294}
\end{barticle}
\endbibitem

\bibitem{CARPENTER199674}
\begin{barticle}
\bauthor{\bsnm{Carpenter}, \binits{M.H.}},
\bauthor{\bsnm{Gottlieb}, \binits{D.}}:
\batitle{Spectral methods on arbitrary grids}.
\bjtitle{Journal of Computational Physics}
\bvolume{129}(\bissue{1}),
\bfpage{74}--\blpage{86}
(\byear{1996}).
\doiurl{10.1006/jcph.1996.0234}
\end{barticle}
\endbibitem

\bibitem{doi:10.1137/120890144}
\begin{barticle}
\bauthor{\bsnm{Gassner}, \binits{G.J.}}:
\batitle{A skew-symmetric discontinuous {G}alerkin spectral element
  discretization and its relation to sbp-sat finite difference methods}.
\bjtitle{SIAM Journal on Scientific Computing}
\bvolume{35}(\bissue{3}),
\bfpage{1233}--\blpage{1253}
(\byear{2013}).
\doiurl{10.1137/120890144}
\end{barticle}
\endbibitem

\bibitem{doi:10.1137/140992205}
\begin{barticle}
\bauthor{\bsnm{Del Rey~Fern\'{a}ndez}, \binits{D.C.}},
\bauthor{\bsnm{Zingg}, \binits{D.W.}}:
\batitle{Generalized summation-by-parts operators for the second derivative}.
\bjtitle{SIAM Journal on Scientific Computing}
\bvolume{37}(\bissue{6}),
\bfpage{2840}--\blpage{2864}
(\byear{2015}).
\doiurl{10.1137/140992205}
\end{barticle}
\endbibitem

\bibitem{DelReyFernandez2018}
\begin{barticle}
\bauthor{\bsnm{DelRey~Fern{\'a}ndez}, \binits{D.C.}},
\bauthor{\bsnm{Hicken}, \binits{J.E.}},
\bauthor{\bsnm{Zingg}, \binits{D.W.}}:
\batitle{Simultaneous approximation terms for multi-dimensional
  summation-by-parts operators}.
\bjtitle{Journal of Scientific Computing}
\bvolume{75}(\bissue{1}),
\bfpage{83}--\blpage{110}
(\byear{2018}).
\doiurl{10.1007/s10915-017-0523-7}
\end{barticle}
\endbibitem

\bibitem{DELREYFERNANDEZ2014214}
\begin{barticle}
\bauthor{\bsnm{{Del Rey Fernández}}, \binits{D.C.}},
\bauthor{\bsnm{Boom}, \binits{P.D.}},
\bauthor{\bsnm{Zingg}, \binits{D.W.}}:
\batitle{A generalized framework for nodal first derivative summation-by-parts
  operators}.
\bjtitle{Journal of Computational Physics}
\bvolume{266},
\bfpage{214}--\blpage{239}
(\byear{2014}).
\doiurl{10.1016/j.jcp.2014.01.038}
\end{barticle}
\endbibitem

\bibitem{doi:10.1137/15M1038360}
\begin{barticle}
\bauthor{\bsnm{Hicken}, \binits{J.E.}},
\bauthor{\bsnm{Del Rey~Fern\'{a}ndez}, \binits{D.C.}},
\bauthor{\bsnm{Zingg}, \binits{D.W.}}:
\batitle{Multidimensional summation-by-parts operators: General theory and
  application to simplex elements}.
\bjtitle{SIAM Journal on Scientific Computing}
\bvolume{38}(\bissue{4}),
\bfpage{1935}--\blpage{1958}
(\byear{2016}).
\doiurl{10.1137/15M1038360}
\end{barticle}
\endbibitem

\bibitem{doi:10.2514/6.2019-3206}
\begin{botherref}
\oauthor{\bsnm{Hicken}, \binits{J.E.}}:
Entropy-Stable, High-Order Discretizations Using Continuous Summation-By-Parts
  Operators.
\doiurl{10.2514/6.2019-3206}
\end{botherref}
\endbibitem

\bibitem{Hicken2020}
\begin{barticle}
\bauthor{\bsnm{Hicken}, \binits{J.E.}}:
\batitle{Entropy-stable, high-order summation-by-parts discretizations without
  interface penalties}.
\bjtitle{Journal of Scientific Computing}
\bvolume{82}(\bissue{2}),
\bfpage{50}
(\byear{2020}).
\doiurl{10.1007/s10915-020-01154-8}
\end{barticle}
\endbibitem

\bibitem{olsson2024projections}
\begin{botherref}
\oauthor{\bsnm{Olsson}, \binits{P.}}:
Projections, Embeddings and Stability
(2024)
\end{botherref}
\endbibitem

\bibitem{Kozdon2016}
\begin{barticle}
\bauthor{\bsnm{Kozdon}, \binits{J.E.}},
\bauthor{\bsnm{Wilcox}, \binits{L.C.}}:
\batitle{Stable coupling of nonconforming, high-order finite difference
  methods}.
\bjtitle{SIAM J. Sci. Comput.}
\bvolume{38}(\bissue{2}),
\bfpage{923}--\blpage{952}
(\byear{2016}).
\doiurl{10.1137/15M1022823}
\end{barticle}
\endbibitem

\bibitem{doi:10.1137/22M1530690}
\begin{barticle}
\bauthor{\bsnm{Wang}, \binits{S.}},
\bauthor{\bsnm{Kreiss}, \binits{G.}}:
\batitle{A finite difference–discontinuous {G}alerkin method for the wave
  equation in second order form}.
\bjtitle{SIAM Journal on Numerical Analysis}
\bvolume{61}(\bissue{4}),
\bfpage{1962}--\blpage{1988}
(\byear{2023}).
\doiurl{10.1137/22M1530690}
\end{barticle}
\endbibitem

\bibitem{ERIKSSON2023111907}
\begin{barticle}
\bauthor{\bsnm{Eriksson}, \binits{G.}},
\bauthor{\bsnm{Werpers}, \binits{J.}},
\bauthor{\bsnm{Niemel\"{a}}, \binits{D.}},
\bauthor{\bsnm{Wik}, \binits{N.}},
\bauthor{\bsnm{Zethrin}, \binits{V.}},
\bauthor{\bsnm{Mattsson}, \binits{K.}}:
\batitle{Boundary and interface methods for energy stable finite difference
  discretizations of the dynamic beam equation}.
\bjtitle{J. Comput. Phys.}
\bvolume{476},
\bfpage{111907}--\blpage{20}
(\byear{2023}).
\doiurl{10.1016/j.jcp.2023.111907}
\end{barticle}
\endbibitem

\bibitem{Eriksson2023}
\begin{barticle}
\bauthor{\bsnm{Eriksson}, \binits{G.}}:
\batitle{Non-conforming interface conditions for the second-order wave
  equation}.
\bjtitle{J. Sci. Comput.}
\bvolume{95}(\bissue{3}),
\bfpage{92}--\blpage{20}
(\byear{2023}).
\doiurl{10.1007/s10915-023-02218-1}
\end{barticle}
\endbibitem

\bibitem{MATTSSON20082293}
\begin{barticle}
\bauthor{\bsnm{Mattsson}, \binits{K.}},
\bauthor{\bsnm{Sv\"{a}rd}, \binits{M.}},
\bauthor{\bsnm{Shoeybi}, \binits{M.}}:
\batitle{Stable and accurate schemes for the compressible {N}avier-{S}tokes
  equations}.
\bjtitle{J. Comput. Phys.}
\bvolume{227}(\bissue{4}),
\bfpage{2293}--\blpage{2316}
(\byear{2008}).
\doiurl{10.1016/j.jcp.2007.10.018}
\end{barticle}
\endbibitem

\bibitem{Almquist2014}
\begin{barticle}
\bauthor{\bsnm{Almquist}, \binits{M.}},
\bauthor{\bsnm{Karasalo}, \binits{I.}},
\bauthor{\bsnm{Mattsson}, \binits{K.}}:
\batitle{Atmospheric sound propagation over large-scale irregular terrain}.
\bjtitle{J. Sci. Comput.}
\bvolume{61}(\bissue{2}),
\bfpage{369}--\blpage{397}
(\byear{2014}).
\doiurl{10.1007/s10915-014-9830-4}
\end{barticle}
\endbibitem

\bibitem{Wang2019}
\begin{barticle}
\bauthor{\bsnm{Wang}, \binits{S.}},
\bauthor{\bsnm{Petersson}, \binits{N.A.}}:
\batitle{Fourth order finite difference methods for the wave equation with mesh
  refinement interfaces}.
\bjtitle{SIAM J. Sci. Comput.}
\bvolume{41}(\bissue{5}),
\bfpage{3246}--\blpage{3275}
(\byear{2019}).
\doiurl{10.1137/18M1211465}
\end{barticle}
\endbibitem

\bibitem{Mattsson2008}
\begin{barticle}
\bauthor{\bsnm{Mattsson}, \binits{K.}},
\bauthor{\bsnm{Ham}, \binits{F.}},
\bauthor{\bsnm{Iaccarino}, \binits{G.}}:
\batitle{Stable and accurate wave-propagation in discontinuous media}.
\bjtitle{J. Comput. Phys.}
\bvolume{227}(\bissue{19}),
\bfpage{8753}--\blpage{8767}
(\byear{2008}).
\doiurl{10.1016/j.jcp.2008.06.023}
\end{barticle}
\endbibitem

\bibitem{Mattsson2006}
\begin{barticle}
\bauthor{\bsnm{Mattsson}, \binits{K.}},
\bauthor{\bsnm{Nordstr\"{o}m}, \binits{J.}}:
\batitle{High order finite difference methods for wave propagation in
  discontinuous media}.
\bjtitle{J. Comput. Phys.}
\bvolume{220}(\bissue{1}),
\bfpage{249}--\blpage{269}
(\byear{2006}).
\doiurl{10.1016/j.jcp.2006.05.007}
\end{barticle}
\endbibitem

\bibitem{eriksson2023acoustic}
\begin{botherref}
\oauthor{\bsnm{Eriksson}, \binits{G.}},
\oauthor{\bsnm{Stiernström}, \binits{V.}}:
Acoustic shape optimization using energy stable curvilinear finite differences
(2023)
\end{botherref}
\endbibitem

\bibitem{engquist}
\begin{barticle}
\bauthor{\bsnm{Engquist}, \binits{B.}},
\bauthor{\bsnm{Majda}, \binits{A.}}:
\batitle{Absorbing boundary conditions for the numerical simulation of waves}.
\bjtitle{Math. Comp.}
\bvolume{31}(\bissue{139}),
\bfpage{629}--\blpage{651}
(\byear{1977}).
\doiurl{10.2307/2005997}
\end{barticle}
\endbibitem

\bibitem{Mattsson2009}
\begin{barticle}
\bauthor{\bsnm{Mattsson}, \binits{K.}},
\bauthor{\bsnm{Ham}, \binits{F.}},
\bauthor{\bsnm{Iaccarino}, \binits{G.}}:
\batitle{Stable boundary treatment for the wave equation on second-order form}.
\bjtitle{J. Sci. Comput.}
\bvolume{41}(\bissue{3}),
\bfpage{366}--\blpage{383}
(\byear{2009}).
\doiurl{10.1007/s10915-009-9305-1}
\end{barticle}
\endbibitem

\bibitem{PETERSSON2016532}
\begin{barticle}
\bauthor{\bsnm{Petersson}, \binits{N.A.}},
\bauthor{\bsnm{O'Reilly}, \binits{O.}},
\bauthor{\bsnm{Sj\"{o}green}, \binits{B.}},
\bauthor{\bsnm{Bydlon}, \binits{S.}}:
\batitle{Discretizing singular point sources in hyperbolic wave propagation
  problems}.
\bjtitle{J. Comput. Phys.}
\bvolume{321},
\bfpage{532}--\blpage{555}
(\byear{2016}).
\doiurl{10.1016/j.jcp.2016.05.060}
\end{barticle}
\endbibitem

\bibitem{hohqmesh}
\begin{botherref}
\oauthor{\bsnm{Kopriva}, \binits{D.A.}},
\oauthor{\bsnm{Ranocha}, \binits{H.}},
\oauthor{\bsnm{Schlottke-Lakemper}, \binits{M.}},
\oauthor{\bsnm{Winters}, \binits{A.}}:
{HOHQMesh.jl}.
\url{https://github.com/trixi-framework/HOHQMesh.jl}
\end{botherref}
\endbibitem

\end{thebibliography}

\end{document}